\documentclass[journal]{new-aiaa}
\usepackage{graphicx} 
\usepackage[rightcaption]{sidecap}
\usepackage{wrapfig}
\usepackage{float}
\usepackage{imakeidx}
\usepackage{comment}
\usepackage{commath}
\usepackage[titletoc]{appendix}
\usepackage{resizegather}
\usepackage[english]{babel}
\usepackage{tabu}
\usepackage{booktabs}
\usepackage{xfrac}
\usepackage{tabularx}

\usepackage{amsmath}
\usepackage{amsthm}
\usepackage{commath}
\usepackage{graphicx,bm}
\usepackage{verbatim}
\usepackage{lscape}
\usepackage{relsize}
\usepackage{enumitem}
\usepackage{textcomp}
\usepackage{breqn}
\usepackage{makecell}
\usepackage{longtable,tabularx}
\usepackage{multirow}
\usepackage{doi}
\usepackage{fancyhdr}
\usepackage{algorithm}
\usepackage{algpseudocode}
\usepackage{setspace}
\usepackage{footnote}
\PassOptionsToPackage{hyphens}{url}
\usepackage{hyperref}
\hypersetup{colorlinks,linkcolor={blue},citecolor={blue},urlcolor={blue}}
\usepackage[numbers]{natbib}
\usepackage{mathtools}
\usepackage[disable]{todonotes}
\usepackage[framemethod=tikz]{mdframed}
\usepackage{booktabs,xcolor,siunitx}
\usepackage{soul}
\soulregister\si7
\soulregister\ref7
\soulregister\cref7
\soulregister\cite7
\soulregister\eqref7
\usepackage{cleveref}
\usepackage{breqn}
\usepackage{todonotes}
\newtheorem{theorem}{Theorem}[section]
\newcommand{\ra}[1]{\renewcommand{\arraystretch}{#1}}

\newcommand{\overbar}[1]{\mkern 1.5mu\overline{\mkern-1.5mu#1\mkern-1.5mu}\mkern 1.5mu}
\newcommand{\overtilde}[1]{\mkern 1.5mu\tilde{\mkern-1.5mu#1\mkern-1.5mu}\mkern 1.5mu}
\newcommand{\f}[2]{\frac{#1}{#2}}
\newcommand{\mb}[1]{\mathbf{#1}}

\DeclareMathAlphabet\mathbfcal{OMS}{cmsy}{b}{n}
\renewcommand{\d}{\mathop{}\!\mathrm{d}} 
\newcommand{\p}{\partial}
\newcommand{\flyp}{f_{\mathrm{lyp}}}
\newcommand{\bflyp}{\bar{f}_{\mathrm{lyp}}}

\title{Adjoint-based Hopf-bifurcation Instability Suppression via \\ First Lyapunov Coefficient}
\author{Sicheng He\footnote{Assistant Professor, AIAA member, sicheng@utk.edu}}
\author{Max Howell\footnote{Graduate Student, mhowel30@vols.utk.edu}}
\affil{University of Tennessee, Knoxville, TN 37996}
\author{Daning Huang\footnote{Assistant Professor, AIAA member, daning@psu.edu}}
\affil{The Pennsylvania State University, University Park, PA 16801}
\author{Eirikur Jonsson\footnote{Research Fellow, AIAA member, eirikurj@umich.edu}}
\affil{University of Michigan, Department of Aerospace Engineering, Ann Arbor, MI 48109}
\author{Galen W. Ng\footnote{Ph.D. Candidate, AIAA member, nggw@umich.edu}}
\affil{University of Michigan, Department of Naval Architecture and Marine Engineering, Ann Arbor, MI 48109}
\author{Joaquim R. R. A. Martins\footnote{Pauline M. Sherman Collegiate Professor, AIAA Fellow, jrram@umich.edu}}
\affil{University of Michigan, Department of Aerospace Engineering, Ann Arbor, MI 48109}

\date{\today}

\definecolor{rev1}{HTML}{FF999A} 
\definecolor{rev2}{HTML}{F3F298} 
\definecolor{rev3}{HTML}{B2E0AE} 
\definecolor{other}{HTML}{C8C7FF} 

\begin{document}
\maketitle{}

\begin{abstract}
Many physical systems exhibit limit cycle oscillations induced by Hopf bifurcations.
{In aerospace engineering, limit cycle oscillations arise from undesirable Hopf-bifurcation phenomena such as aeroelastic flutter and transonic buffet}. 
{In some cases, the resulting limit cycle oscillations can themselves be unstable, leading to amplitude divergence or hysteretic transitions that threaten structural integrity and performance.}
{Avoiding such phenomena when performing gradient-based design optimization requires a constraint that quantifies the stability of the bifurcations and the derivative of that constraint with respect to the design variables.}
To capture the {local} stability of bifurcations, we leverage the first Lyapunov coefficient, {which predicts whether the resulting limit cycle oscillation is stable or unstable.}
We develop an accurate and efficient method for computing derivatives of the first Lyapunov coefficient.
{We leverage the adjoint method and reverse algorithmic differentiation to efficiently compute the derivative of the first Lyapunov coefficient.}
We demonstrate the efficacy of the proposed adjoint method in {three} design optimization problems that suppress {unstable bifurcation}: {an algebraic Hopf bifurcation model,}
{an aeroelastic model of a typical section}, and {a nonlinear problem based on the complex Ginzburg--Landau partial differential equation}. 
{While the current formulation addresses only a single bifurcation mode}, the proposed adjoint shows great potential for efficiently handling Hopf bifurcation constraints in large-scale nonlinear problems governed by partial differential equations.
Its accuracy, versatility and scalability make it a promising tool for aeroelastic and aerodynamic design optimization as well as other engineering problems involving Hopf bifurcation instabilities.
\end{abstract}

\section{Introduction}

Bifurcation is the shift from a linearly stable to a linearly unstable solution due to a single parameter change.
This occurrence involves several mechanisms, including the saddle node, pitchfork, transcritical, and Hopf bifurcation~\cite{Kuznetsov2004}.
The Hopf-bifurcation is frequently encountered in phenomena relevant to aircraft design.
Examples of such phenomena include flutter, laminar-turbulent transition, Kármán vortex street, and transonic buffet.
Besides the bifurcation parameter, another essential characteristic of the bifurcation is its stability.
A stable Hopf-bifurcation leads to a supercritical limit cycle oscillation (LCO), which is a benign response.
In contrast, an unstable Hopf-bifurcation leads to a subcritical LCO, which can be bi-stable and exhibit a hysteretic response.
In addition, bifurcation can lead to semi-stable LCOs.
Unstable Hopf-bifurcations should be eliminated from safe designs.
Subcritical bifurcations are frequently encountered {experimentally} in flutter~\cite{Tang2001,Drachinsky2022}.
In this paper, we develop an accurate and efficient method for integrating a constraint on the stability of Hopf bifurcations into gradient-based design optimization using the first Lyapunov coefficient and an adjoint formulation.

Significant progress has been achieved in tackling optimization problems constrained by ordinary differential equations (ODEs) and partial differential equations (PDEs).
In the aerospace industry, various types of dynamical system solutions have been explored and applied in design optimization.
{Equilibrium point solutions have been explored}~\cite{Jameson1988,Kenway2014c}, {however these problems are limited because cases involving non-stationary operating conditions frequently involve more complex dynamical system solutions.
To address this issue, solutions of problems involving bifurcations}~\cite{Marquet2008,Shi2020a,Shi2023a,Church2019,Boulle2022a,Boull2023}, {LCOs}~\cite{Thomas2020,Prasad2022,He2022b,He2021c,Riso2023,Golla2024}, {and chaotic dynamics}~\cite{Wang2014,RepolhoCagliari2021} {have been explored in recent research.
These solutions can be applied to complex cases such as aeroelastic flutter and transonic buffet, which are Hopf-bifurcations characterized by a limit-cycle oscillation (LCO) in the post-flutter response.} 

To increase energy efficiency, next-generation aircraft wings and large floating offshore wind turbines will have larger aspect ratios and be more flexible~\cite{Afonso2017}.
Such designs are more likely to exhibit undesirable dynamic behavior at {non-stationary operating conditions}.
Thus, it is critical to develop methods to efficiently and accurately consider the bifurcation constraints in a multidisciplinary design optimization (MDO) framework.
Furthermore, MDO problems, such as aerodynamic and aerostructural optimization, usually involve a large number of design variables, requiring the use of gradient-based optimization~\cite{Lyu2014f,Kenway2014c}.
Such optimizations rely on efficient gradient computation algorithms that scale well with the number of design variables~\cite[Ch. 6]{Martins2022}.
Hence, the challenges for bifurcation-constrained optimization are two-fold: (1) how to quantify the stability of bifurcation? (2) how to compute the gradient of such a quantity?

\begin{figure}[h]
\centering
\includegraphics[angle=0,width=0.7\linewidth]{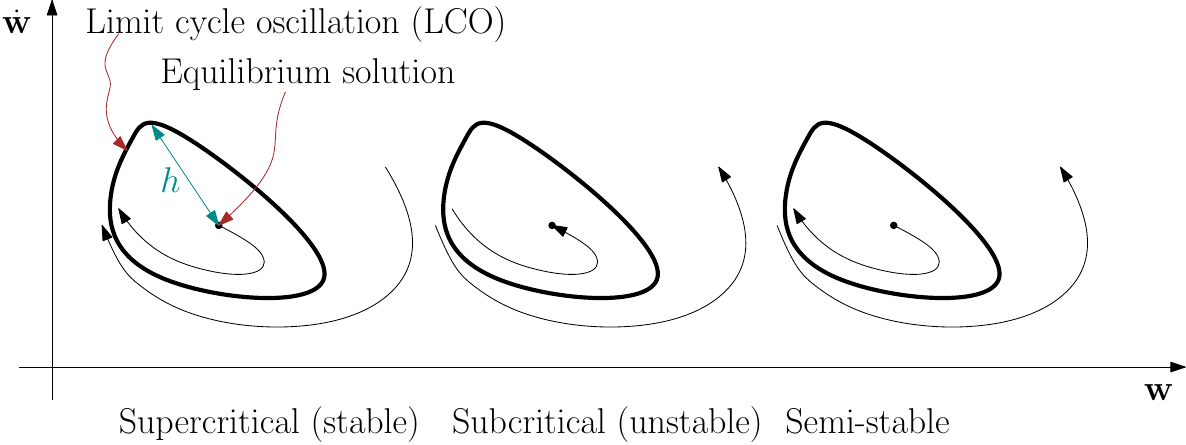}
\caption{Stable, unstable, and semi-stable Hopf bifurcations. For each case, the motion magnitude, $h$, approaches the limit $h \rightarrow 0$ to measure the bifurcation stability at the equilibrium point. The direction of motion for solutions in the vicinity of the limit cycle oscillation (LCO) is indicated by the arrows.}
\label{fig:sketch}
\end{figure}

The stability of Hopf-bifurcation can be assessed through three primary categories of methods:
(1) methods based on Floquet theory~\citep{Friedmann1977,Chicone2006}, (2) methods based on the LCO curve, i.e., LCO {amplitude}-LCO parameter response curve~\citep{He2023a,Riso2020,Riso2023a}, and (3) center manifold-based 
methods~\citep{Kuznetsov2004,Stanford2013a,Habib2015,Gai2016a,Malher2017a}.
The Floquet theory- and LCO curve-based method use a similar strategy because both analyze the post-bifurcation response, and bifurcation stability is determined by considering motion magnitude approaching zero, as shown in Fig.~\ref{fig:sketch}, {where the resulting supercritical LCO converges to an equilibrium solution}.

Floquet theory-based methods posit that the stability of a periodic dynamical system depends on the eigenvalue distribution of the monodromy matrix.
This matrix is obtained by evaluating the fundamental solution matrix of the periodic system at times $t=0$ and $t=T$, where $T$ represents one period of the system \cite[Ch.~2.4]{Chicone2006}.
In practice, acquiring the monodromy matrix involves simulating the time-periodic dynamical system $N$ times over one period, where $N$ is the number of states of the dynamical system.
However, the computational cost of this process is high for problems governed by PDEs, particularly when $N$ is large, as seen in high-fidelity aeroelastic systems studied by \citet{Thomas2002nx} and \citet{He2019b}. 
In response to this challenge, \citet{Bauchau2001} proposed using the Arnoldi algorithm to compute the dominant eigenvalue instead of calculating all eigenvalues to assess the stability of a periodic dynamical system.
This approach is appealing because it avoids explicit matrix formation.
Instead, the approach requires only matrix-vector products through the Arnoldi algorithm, where the response vector can be evaluated using time integration for one period with a specified initial condition.
However, this approach requires an unsteady adjoint for derivative computation, which requires a substantial implementation effort and is computationally intensive~\citep{Cao2003}.
\citet{Seyranian2000} introduced forward mode differentiation of Floquet analysis to bypass the need for an unsteady adjoint.
However, the computational cost of this derivative computation method scales with the number of design variables, which is undesirable for large-scale MDO.

The LCO curve-based methods leverage the fact that in an LCO {amplitude}-LCO parameter response curve, the region that has a negative slope is unstable.
By studying the stability of an LCO with a small motion magnitude, this class of methods can be used to study bifurcation stability.
Leveraging the earlier work of time spectral-based LCO equation by \citet{Thomas2002nx}, and its adjoint counterpart by \citet{He2022b}, \citet{He2023a} proposed using a polynomial fit of the LCO curve to study its stability.
The adjoint method is used to efficiently and accurately compute the derivatives.
{However, only one flight condition and mode shape was used in the analysis, which may not account for all flutter mechanisms.}
\citet{Riso2021a,Riso2023a} proposed an approximate recovery rate as a metric for post-flutter analysis and gradient-based optimization, avoiding the need for expensive transient simulations.
While successful in eliminating subcritical characteristics and accounting for mode switches and hump modes, {the approximation of the recovery rate metric may lead to a conservative design.
Furthermore, the adjoint derivatives were not developed, potentially limiting its application.} 
Besides the aforementioned methods, other LCO computation methods like the continuation method~\citep{Dimitriadis2008,Shukla2017b} can also be extended to study bifurcation stability leveraging the LCO curve.
{However, the continuation method requires iterative solvers to converge to a solution at each point on the LCO curve, making gradient based optimization expensive. 
Therefore, this method is outside the scope of this study.}

The center manifold-based methods posit that the stability of a solution in a high-dimensional dynamical system near the bifurcation point can be characterized by a low-dimensional manifold known as the center manifold~\citep{Kuznetsov2004,Stanford2013a,Habib2015,Gai2016a,Malher2017a}.
The center manifold for all Hopf bifurcations has a dimension of two regardless of the system's state dimension.
This two-dimensional center manifold corresponds to a purely imaginary conjugate pair of eigenvalues.
The remaining dimensions correspond to eigenvalues with negative real parts which are ultimately damped. 
Using the Jacobian, Hessian, and third-order derivative matrices (tensors) computed on the center manifold, we could determine the stability of a bifurcation by the first Lyapunov coefficient. 
{These higher order derivatives may be difficult to compute for higher dimensional problems.
However, these terms can be estimated using second-order accurate finite differences}~\cite{Beran1999} {and scaled up to finite element method problems}~\cite{Stanford2013a} {to efficiently calculate the first Lyapunov coefficient for large problems.}
The first Lyapunov coefficient is positive for subcritical bifurcations (unstable) and negative for supercritical bifurcations (stable).
A zero first Lyapunov coefficient cannot be used to determine the bifurcation stability (indeterministic); in this case, higher-order Lyapunov coefficients are needed to determine stability.
When compared to the LCO curve-based methods, the key advantage of center manifold-based methods is that the latter only requires information at one bifurcation parameter (e.g., one flight condition).
Center manifold-based methods have been explored in the literature.
For example, \citet{Stanford2013a} utilized a multi-scale method to capture the bifurcation diagram near the bifurcation point and conducted optimization to mitigate the subcritical response using {finite differences} for derivative computation.
However, the computational cost of the {finite difference} method is proportional to the number of design variables, which is undesirable for large-scale MDO.

To overcome the scalability challenge for bifurcation-constrained MDO problems, we develop a Hopf-bifurcation stability analysis based on the first Lyapunov coefficient and a scalable implementation of its gradient computation.
The proposed method builds on a three-level stability analysis framework, {which operates in the outlined modules} in Fig.~\ref{fig:Hopf_schematic}. 
First, the linear stability analysis solves the Jacobian matrix eigenvalue problem (EVP) of an equilibrium point to determine if the system is stable, unstable, and marginally stable. 
{Next, the Hopf-bifurcation analysis captures the exact onset of the bifurcation point by calculating the critical onset condition and bifurcation frequency.
This module only considers one bifurcation instability mode.
A limitation of this method is that it cannot handle changing bifurcation modes during the optimization.}
Lastly, the Hopf-bifurcation stability analysis captures whether the resulting LCO is stable or not.
This step is a ``local'' method---it can predict the LCO stability in the vicinity of the bifurcation point, but it cannot determine the stability of a large magnitude LCO.
{The three aforementioned steps of assessing Hopf-bifurcation stability are shown in Fig.}~\ref{fig:Hopf_schematic}.
The proposed method utilizes the first two columns: linear stability and bifurcation stability.
{The third column shows the nonlinear stability analysis for post bifurcation trajectories.}

We addressed two challenges to establish this framework. 
First, the first Lyapunov coefficient has a complicated functional form that contains matrix inversion, higher-order partial derivative matrix-vector products, and complex analytic function operation.
This poses a challenge in its derivative computation.
Second, the derivative computation across the three levels involves the complex interaction between the underlying equilibrium point solution and the multiple stability modules~\cite{He2025a}.
This requires carefully designed algorithms to minimize the computational cost.

\begin{figure}[H]
\centering
\includegraphics[angle=0,width=0.8\linewidth]{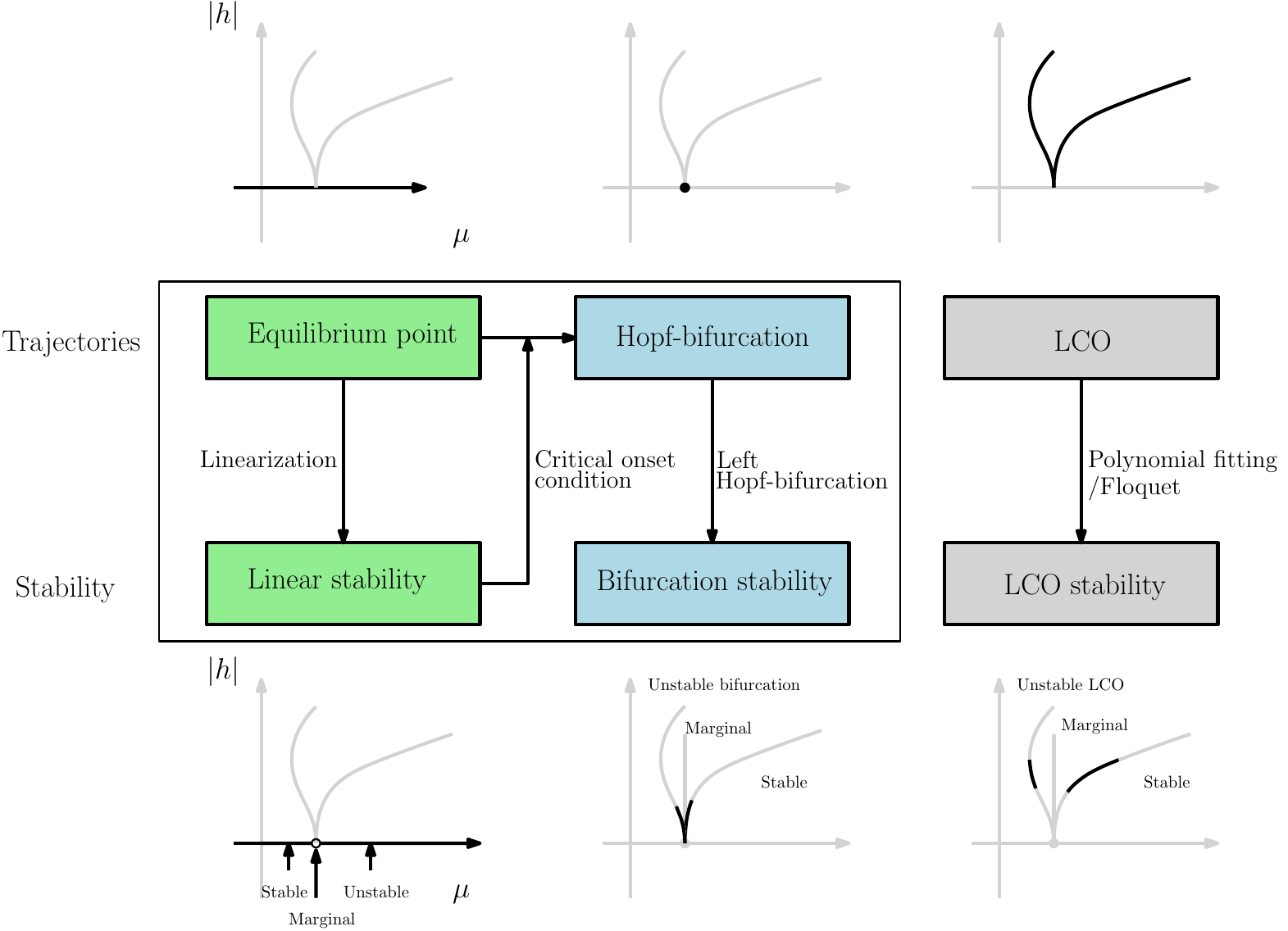}
\caption{{Three levels of stability analysis related to Hopf bifurcation: linear, Hopf bifurcation, and LCO stability. The proposed method builds on the first two columns.}}
\label{fig:Hopf_schematic}
\end{figure}

Flow control, aerodynamic, and aerostructural optimization frequently involve a large number of design variables, necessitating gradient-based optimization~\cite[Ch.~1]{Martins2022}.
It is also crucial to efficiently compute the derivatives of the objective and constraint functions with respect to these numerous design variables.
There are two types of methods whose computational cost scales weakly with the number of design variables: reverse algorithmic differentiation (RAD)~\cite{Hascoet:2004:T2U} and the adjoint method~\cite{Jameson1988}.

To enable the accurate gradient computation of highly complex functions, we can resort to algorithmic differentiation (AD), a well-established technique for systematically computing derivatives using the chain rule~\cite{Hascoet:2004:T2U}.
AD can be divided into two main types: forward AD (FAD) and reverse AD (RAD), which differ in the order in which the chain rule is applied.
FAD computes derivatives in a forward direction, propagating from inputs to outputs, while RAD operates in reverse, starting from outputs and tracing back to inputs.
The computational cost of FAD scales with the number of inputs, whereas that of RAD scales with the number of outputs~\cite[Sec.~6.6.4]{Martins2022}.

AD can be implemented through operator overloading, source code transformation, or by leveraging analytic AD formulae for specific foundational matrix operations such as inversion, products, and computation of eigenvalues and eigenvectors~\cite{Dwyer1948,Giles2008a}. 
Using analytic formulae in AD {enables optimized libraries to compute derivatives} without differentiating the underlying library source code.
\Citet{Giles2008a} proposed a method to derive RAD formulae for real functions from FAD formulae, which are typically easier to obtain.
\Citet{He2022a} extended this method to complex analytic functions.

Beyond explicit analytic methodologies such as AD, implicit analytic methods, including direct and adjoint approaches, are also viable alternatives~\cite[Chapter~6.7]{Martins2022}.
{The direct method is advantageous when the number of outputs exceeds the number of inputs, whereas the adjoint method is more efficient when the inputs outnumber the outputs.}
The adjoint method forms the basis for methods applied in stability prediction for optimization~\cite{Thomas2020,Prasad2022,He2022b,He2023a,Kennedy2014d,Boull2023,Marquet2008,Mettot2014}.
Despite advances in flutter-constrained aircraft design optimization~\cite{Jonsson2019b}, direct methods are less common in optimization due to their computational cost.

In this paper, we leverage RAD and the complex dot product identity tool for efficient RAD formula derivation~\cite{He2023} to compute the partial derivatives of the first Lyapunov coefficient.
We then use these partial derivatives in an efficient adjoint method to compute the Hopf-bifurcation derivative accurately and efficiently with respect to a large number of design variables.
The contributions of the paper are summarized as follows:
(1) We develop an efficient and accurate bifurcation stability derivative computation framework based on the adjoint method; (2) We derive an efficient RAD formula for the first Lyapunov coefficient; (3) We apply the developed adjoint method to a nonlinear aeroelastic {model of a typical airfoil section} {and the nonlinear complex Ginzburg--Landau equation} to suppress subcritical bifurcation.
{It should be noted that the present formulation assumes a single, simple pair of complex-conjugate eigenvalues at the Hopf point and therefore does not yet address cases involving repeated or nearly degenerate eigenvalues, which may require a multi-mode extension of the framework.}

The paper is organized as follows.
In Section~\ref{sec:equilibrium}, we present the governing equations for the equilibrium point and its higher derivatives required for bifurcation stability analysis.
Next, we present the eigenvalue analysis for the bifurcation problem in Section~\ref{sec:eigen}.
Then, we present the first Lyapunov coefficient used to determine the Hopf-bifurcation stability in Section~\ref{sec:lyapunov_coeff}.
We summarize the Hopf stability analysis presented in Section~\ref{sec:analysis}, building on the developments from Sections~\ref{sec:equilibrium}, \ref{sec:eigen}, and~\ref{sec:lyapunov_coeff}.
In Section~\ref{sec:stability_der}, we develop the adjoint method and derive the RAD formula for the first Lyapunov coefficient.
We demonstrate the general optimization problem in Section~\ref{sec:opt}.
In Section~\ref{sec:example}, the proposed method is applied to analyze and optimize bifurcation stability using {an algebraic Hopf-bifurcation example}, a nonlinear aeroelastic flutter test case {of a typical section model}, and {the complex Ginzburg--Landau partial differential equation}.
We finish with concluding remarks in Section~\ref{sec:con}.

\section{Equilibrium point and Taylor series expansion}
\label{sec:equilibrium}

This section presents the governing equations for the dynamical system and their Taylor series expansion around the equilibrium point.
Differentiation of these equations yields the Jacobian, Hessian, and third-order derivative matrices that are essential for formulating the eigenvalue problems in Section~\ref{sec:eigen} and evaluating the first Lyapunov coefficient in Section~\ref{sec:lyapunov_coeff}. 

\subsection{Equilibrium point}
The dynamical system can be written as
\begin{equation}
\label{eq:nonlinear_dyn}
\frac{\d\mb{w}}{\d t} = \mb{r}(\mb{w}, \mu, \mb{x}),
\end{equation}
where $\mb{r}:\mathbb{R}^{n}\times\mathbb{R}\times\mathbb{R}^{n_{\mb{x}}} \rightarrow \mathbb{R}^{n}$ is some nonlinear function, $\mb{w}\in\mathbb{R}^{n}$ is the state variable vector, $\mu\in \mathbb{R}$ is the bifurcation parameter, and $\mb{x}\in\mathbb{R}^{n_{\mb{x}}}$ is the design variable vector.
The bifurcation parameter, $\mu$, determines the onset of the bifurcation, while the critical value at which the bifurcation occurs is not known in advance.
The bifurcation can be characterized by various physical parameters, such as flow speed in flutter and angle-of-attack in transonic buffet.
The critical bifurcation parameter is determined by the design variable vector, $\mb{x}$.
The equilibrium solution $\overtilde{\mb{w}}\in\mathbb{R}^n$ is obtained by solving the following steady-state equation for a given $\mu$ and $\mb{x}$: 
\begin{equation}
\label{eq:equilibrium}
\mb{r}(\overtilde{\mb{w}}, \mu, \mb{x}) = 0.
\end{equation}
The steady-state equation solution parametrizes the EVP and the Hopf bifurcation stability, which will be discussed in Section~\ref{sec:eigen}.

\subsection{Taylor series expansion}

For linear stability analysis, we only need to consider the linear term involving the Jacobian matrix.
However, the Hopf-bifurcation stability requires additional second- and third-order partial derivative terms.
We perform a Taylor series expansion around an equilibrium solution.
The state, $\mb{w}$, can be decomposed into an equilibrium solution and a perturbed state $\delta \mb{w}$ as follows~\cite{Li2023a}:
\begin{equation}
\mb{w} = \overtilde{\mb{w}} + \delta\mb{w}.
\end{equation}
Then, Eq.~\ref{eq:nonlinear_dyn} can be written as
\begin{equation}
\label{eq:nonlinear_dyn_decomp}
\begin{aligned}
\f{\d \overtilde{\mb{w}} + \delta\mb{w}}{\d t}
= & \mb{r}(\overtilde{\mb{w}} + \delta\mb{w}, \mu, \mb{x})\\
\f{\d \delta\mb{w}}{\d t}
= & \mb{r}(\overtilde{\mb{w}}, \mu, \mb{x}) + \mb{A}(\overtilde{\mb{w}}, \mu, \mb{x})\delta \mb{w} + \f{1}{2}\mb{b}(\delta \mb{w}, \delta \mb{w}, \overtilde{\mb{w}}, \mu, \mb{x}) + \f{1}{6} \mb{c}(\delta \mb{w}, \delta \mb{w}, \delta \mb{w}, \overtilde{\mb{w}}, \mu, \mb{x}) + \mathcal{O}\left(\norm{\delta\mb{w}}^4 \right) \\
\f{\d \delta\mb{w}}{\d t}
= & \mb{A}(\overtilde{\mb{w}}, \mu, \mb{x})\delta \mb{w} + \f{1}{2}\mb{b}(\delta \mb{w}, \delta \mb{w},\overtilde{\mb{w}}, \mu, \mb{x}) + \f{1}{6} \mb{c}(\delta \mb{w}, \delta \mb{w}, \delta \mb{w}, \overtilde{\mb{w}}, \mu, \mb{x}) + \mathcal{O}\left(\norm{\delta\mb{w}}^4 \right)                                            \\
\end{aligned},
\end{equation}
where we substituted Eq.~\ref{eq:equilibrium} to simplify $\d \left( \delta \mb{w} \right)/\d t$, $\mb{A}\in\mathbb{R}^{n\times n}$ is the Jacobian of $\mb{r}$, and $\mb{b}$ and $\mb{c}$ are the Hessian and third-order partial derivative matrices, respectively.
The $i$th entries of $\mb{b}$ and $\mb{c}$ are
\begin{equation}
\label{eq:bc}
\begin{aligned}
\mb{b}_i(\mb{y}_1, \mb{y}_2, \overtilde{\mb{w}}, \mu, \mb{x}) =           & \sum_{j,k=1}^{n}\f{\p^2 \mb{r}_i}{\p \mb{w}_j \p \mb{w}_k}\left(\overtilde{\mb{w}}, \mu, \mb{x}\right)  \mb{y}_{1,j} \mb{y}_{2,k},                             \\
\mb{c}_i(\mb{y}_1, \mb{y}_2, \mb{y}_3, \overtilde{\mb{w}}, \mu, \mb{x}) = & \sum_{j, k, l=1}^{n}\f{\p^3 \mb{r}_i}{\p \mb{w}_j \p \mb{w}_k \p \mb{w}_l}\left(\overtilde{\mb{w}}, \mu, \mb{x}\right) \mb{y}_{1,j} \mb{y}_{2,k} \mb{y}_{3,l}, \\
\end{aligned}
\end{equation}
where $i=1, \ldots, n$, and $\mb{y}_1, \mb{y}_2, \mb{y}_3\in\mathbb{R}^{n}$ are arbitrary vectors.

{In more complex problems where analytical solutions do not exist for the higher order derivatives, $\mb{b}$ and $\mb{c}$ can be estimated using second order accurate finite differences as shown by} \citet{Beran1999}. 
{This method only requires a total of six calls to the Jacobian matrix to calculate both the $\mb{b}$ and $\mb{c}$ vectors, reducing computational cost significantly.}
{We implement and demonstrate this approach in the PDE example in} Section~\ref{sec:CGL_example}.

\section{Eigenvalue problems (EVPs) for Hopf bifurcation}
\label{sec:eigen}
This section formulates the right and left EVPs at the Hopf bifurcation that are amenable for subsequent gradient computation.
The left EVP is included because it is required in the computation of the first Lyapunov coefficient (Section~\ref{sec:lyapunov_coeff}).

\subsection{Right eigenvalue problem}

We now derive the equation that governs the dynamical system at the bifurcation point.
At the bifurcation point, the bifurcation parameter, {$\mu_{\mathrm{bif}}$}, and the equilibrium solution, {$\overtilde{\mb{w}}_{\mathrm{bif}}$}, defines a Jacobian matrix, $\mb{A}(\mu_{\mathrm{bif}}, \overtilde{\mb{w}}_{\mathrm{bif}})$, where the following equation is satisfied:
\begin{equation}
\label{eq:bifurcation}
\begin{aligned}
\mb{A} {\mb{q}_R}_{\mathrm{bif}} =    & \lambda_{R,\mathrm{bif}}{\mb{q}_R}_{\mathrm{bif}}, \\
\mathrm{Re}(\lambda_{R, \mathrm{bif}}) = & 0,                                               \\
\end{aligned}
\end{equation}
where $\lambda_{R, \mathrm{bif}} \in\mathbb{C}, {\mb{q}_R}_{\mathrm{bif}} \in\mathbb{C}^{n}$ are the right (indicated by the subscript $\square_{R}$) eigenvalues and eigenvectors, respectively.
{At the bifurcation point, the derivative $\d \mathrm{Re}(\lambda_{R, \mathrm{bif}}) / \d \mu_{\text{bif}}$ must be positive.
This condition can be enforced explicitly when using a root solving method.}
For Hopf-bifurcation, the following two eigenpairs satisfy Eq.~\ref{eq:bifurcation}:
\begin{equation}
(j{\omega}_{\mathrm{bif}}, {\mb{q}_R}_{\mathrm{bif}}), (-j{\omega}_{\mathrm{bif}}, {\mb{q}_R}^*_{\mathrm{bif}}),
\end{equation}
where {${\omega}_{\mathrm{bif}} > 0$} is the bifurcation frequency, the superscript $\square^{*}$ is the conjugate transpose operator, and $j$ denotes the imaginary unit.

Thus, the eigenvalue problem in Eq.~\ref{eq:bifurcation} at the Hopf-bifurcation point can be written as
\begin{equation}
\label{eq:r_eig}
\begin{aligned}
\mb{A}{\mb{q}_R}_{\mathrm{bif}} =                       & j{\omega}_{\mathrm{bif}}{\mb{q}_R}_{\mathrm{bif}}, \\
{\mb{q}_R}_{\mathrm{bif}}^* {\mb{q}_R}_{\mathrm{bif}} = & 1,                                                   \\
\mb{e}_k^\intercal {\mb{q}_R}_{\mathrm{bif}, i} =       & 0,
\end{aligned}
\end{equation}
where $k$ is the index corresponding to the entry with the maximum magnitude, and $\square_{i}$ denotes the operator to extract the imaginary part of the complex variable, ${\mb{q}_R}_{\mathrm{bif}, i} = \mathrm{Im}\left({\mb{q}_R}_{\mathrm{bif}}\right)$.
The index $k$ is defined by
\begin{equation}
k = \mathrm{argmax}_{l=1, \ldots, n} \norm{{\mb{q}_R}_{\mathrm{bif}, i, l}},
\end{equation}
where ${\mb{q}_R}_{\mathrm{bif}, i, l}$ is the $l^{\mathrm{th}}$ entry of the eigenvector ${\mb{q}_R}_{\mathrm{bif}, i}$.
The second and third rows in Eq.~\ref{eq:r_eig} are magnitude and phase constraints, which make the solution unique.
Without these two constraints, the eigenvector can stretch and rotate in the complex plane, {making the derivatives calculated in subsequent modules ambiguous.
Also, the magnitude constraint adds numerical stability because it normalizes the eigenvectors.}
\citet{He2023} derive the eigenvalue adjoint equations in more detail.

We can write Eq.~\ref{eq:r_eig} purely in terms of real numbers by expanding the complex equations into two real equations.
The state variable at the bifurcation point satisfies the following set of equations:
\begin{equation}
\label{eq:REVP}
\mb{r}_{R}(\mb{u}_R, \mb{x}) =
\begin{bmatrix}
\mb{r}({\mu}_{\mathrm{bif}}, {\overtilde{\mb{w}}}_{\mathrm{bif}}, \mb{x})                                                                                      \\
\mb{A}({\mu}_{\mathrm{bif}}, {\overtilde{\mb{w}}}_{\mathrm{bif}}, \mb{x}) {\mb{q}_R}_{\mathrm{bif}, r} + {\omega}_{\mathrm{bif}}{\mb{q}_R}_{\mathrm{bif}, i} \\
\mb{A}({\mu}_{\mathrm{bif}}, {\overtilde{\mb{w}}}_{\mathrm{bif}}, \mb{x}) {\mb{q}_R}_{\mathrm{bif}, i} - {\omega}_{\mathrm{bif}}{\mb{q}_R}_{\mathrm{bif}, r} \\
{\mb{q}_R}_{\mathrm{bif}, r}^\intercal {\mb{q}_R}_{\mathrm{bif}, r} + {\mb{q}_R}_{\mathrm{bif}, i}^\intercal {\mb{q}_R}_{\mathrm{bif}, i} - 1                      \\
\mb{e}_j^\intercal {\mb{q}_R}_{\mathrm{bif}, i}
\end{bmatrix}
= \mb{0}, \quad
\mb{u}_R =
\begin{bmatrix}
{\overtilde{\mb{w}}}_{\mathrm{bif}} \\
{\mb{q}_R}_{\mathrm{bif}, r}          \\
{\mb{q}_R}_{\mathrm{bif}, i}          \\
{\mu}_{\mathrm{bif}}                \\
{\omega}_{\mathrm{bif}}             \\
\end{bmatrix},
\end{equation}
where $\mb{r}_R$ denotes the right eigenvalue problem in residual form and $\mb{u}_R$ is the combined state vector.
The first equation is the equilibrium solution, and the other four equations are the right EVP for the bifurcation.

\subsection{Left eigenvalue problem}
The left eigenvalue problem is defined as,
\begin{equation}
\label{eq:l_eig}
\begin{aligned}
\mb{A}^\intercal {\mb{q}_L}_{\mathrm{bif}} =            & - j {\omega}_{\mathrm{bif}}{\mb{q}_L}_{\mathrm{bif}}, \\
{\mb{q}_R}_{\mathrm{bif}}^* {\mb{q}_L}_{\mathrm{bif}} = & 1,                                                      \\
\end{aligned}
\end{equation}
where ${\mb{q}_L}_{\mathrm{bif}}$ is one left eigenvector, and ${\mb{q}_L}_{\mathrm{bif}, r}$, and ${\mb{q}_L}_{\mathrm{bif}, i}$ are the real and imaginary parts, respectively.
The second row is a normalization condition.
The choice of this special normalization condition is motivated by the first Lyapunov coefficient discussed in Section~\ref{sec:lyapunov_coeff}.
Similarly to the right eigenvalue problem, we can write Eq.~\ref{eq:l_eig} in real space as
\begin{equation}
\label{eq:LEVP}
\mb{r}_{L}({\mb{u}_L}, \mb{x}, \mb{u}_R)=
\begin{bmatrix}
\mb{r}({\mu}_{\mathrm{bif}}, {\overtilde{\mb{w}}}_{\mathrm{bif}}, \mb{x})                                                                                                \\
\mb{A}^\intercal({\mu}_{\mathrm{bif}}, {\overtilde{\mb{w}}}_{\mathrm{bif}}, \mb{x}) {\mb{q}_L}_{\mathrm{bif}, r} - {\omega}_{\mathrm{bif}}{\mb{q}_L}_{\mathrm{bif}, i} \\
\mb{A}^\intercal({\mu}_{\mathrm{bif}}, {\overtilde{\mb{w}}}_{\mathrm{bif}}, \mb{x}) {\mb{q}_L}_{\mathrm{bif}, i} + {\omega}_{\mathrm{bif}}{\mb{q}_L}_{\mathrm{bif}, r} \\
{\mb{q}_R}_{\mathrm{bif}, r}^\intercal {\mb{q}_L}_{\mathrm{bif}, r} + {\mb{q}_R}_{\mathrm{bif}, i}^\intercal {\mb{q}_L}_{\mathrm{bif}, i} - 1                                \\
- {\mb{q}_R}_{\mathrm{bif}, i}^\intercal {\mb{q}_L}_{\mathrm{bif}, r} + {\mb{q}_R}_{\mathrm{bif}, r}^\intercal {\mb{q}_L}_{\mathrm{bif}, i}                                  \\
\end{bmatrix}
=\mb{0}, \quad
{\mb{u}_L}=
\begin{bmatrix}
{\overtilde{\mb{w}}}_{\mathrm{bif}} \\
{\mb{q}_L}_{\mathrm{bif}, r}          \\
{\mb{q}_L}_{\mathrm{bif}, i}          \\
{\mu}_{\mathrm{bif}}                \\
{\omega}_{\mathrm{bif}}             \\
\end{bmatrix}.
\end{equation}
where $\mb{r}_L$ denotes the left eigenvalue problem in residual form, and $\mb{u}_L$ is the combined state vector.

The bifurcation problem can be formulated using either the right or the left EVP.
The underlying equilibrium point solution, bifurcation parameter, and bifurcation frequency computed using both methods are equal to each other.
Thus, for the proposed stability metric, we need the bifurcation parameter, frequency, and the equilibrium state vector---taken from the right EVP problem solution---and the eigenvectors from both the left and right EVPs.

\section{First Lyapunov coefficient}
\label{sec:lyapunov_coeff}

The first Lyapunov coefficient is a metric that can be used to determine the stability of Hopf bifurcation.
It is defined in Theorem~\ref{thm:lyp} as stated below.
The sign of the first Lyapunov coefficient determines the stability of a Hopf bifurcation solution.
\begin{theorem}
\label{thm:lyp}
The first Lyapunov coefficient, $f_{\text{Lyp}}$, is defined as follows
\begin{equation}
\label{eq:lyp}
f_{\mathrm{lyp}}({\omega}_{\mathrm{bif}}, {\mb{q}_L}_{\mathrm{bif}},{\mb{q}_R}_{\mathrm{bif}}, \mb{A}, \mb{B}, \mb{C}) =  \f{1}{2{\omega}_{\mathrm{bif}}} \mathrm{Re}\left(h_1 - 2 h_2 + h_3\right),
\end{equation}
where
\begin{equation}
\begin{aligned}
h_1 =                                                                                                       & {\mb{q}_L}_{\mathrm{bif}}^* \mb{c}({\mb{q}_R}_{\mathrm{bif}}, {\mb{q}_R}_{\mathrm{bif}}, {\mb{q}_R}^*_{\mathrm{bif}}),                                                            \\
h_2 =                                                                                                       & {\mb{q}_L}_{\mathrm{bif}}^* \mb{b}({\mb{q}_R}_{\mathrm{bif}}, \mb{A}^{-1}\mb{b}({\mb{q}_R}_{\mathrm{bif}}, {\mb{q}_R}^*_{\mathrm{bif}})),                                         \\
h_3 =                                                                                                       & {\mb{q}_L}_{\mathrm{bif}}^* \mb{b}({\mb{q}_R}_{\mathrm{bif}}^*, (2j{\omega}_{\mathrm{bif}}\mb{I}_n - \mb{A})^{-1}\mb{b}({\mb{q}_R}_{\mathrm{bif}}, {\mb{q}_R}_{\mathrm{bif}})). \\
\end{aligned}
\end{equation}
Here, $\mb{B}$ and $\mb{C}$ are tensors that store the coefficient data for $\mb{b}(\cdot)$ and $\mb{c}(\cdot)$, respectively.
Based on the sign of the first Lyapunov coefficient, $f_{\mathrm{lyp}}$, we can determine the stability of the bifurcation:
\begin{equation}
\begin{cases}
\text{Stable,} & \text{if } f_{\mathrm{lyp}}<0\\
\text{Indeterminate} & \text{if } f_{\mathrm{lyp}} = 0 \\ 
\text{Unstable,} & \text{if } f_{\mathrm{lyp}}>0 \\
\end{cases},
\end{equation}
\end{theorem}
\begin{proof}
See \citet[Section 5.4]{Kuznetsov2004}.
\end{proof}

If an eigenvector is scaled (i.e., stretched and rotated in the complex plane) it still remains an eigenvector.
Such scaling should not affect a conclusion on bifurcation stability or, equivalently, the sign of the first Lyapunov coefficient (Eq.~\ref{eq:lyp}). 
{As an additional measure to ensure the sign invariance to scaling in our formulation and also ensure that derivatives are well defined}, we require the following normalization condition:
\begin{equation*}
{\mb{q}_R}_{\mathrm{bif}}^* {\mb{q}_L}_{\mathrm{bif}} = 1,
\end{equation*}
which is proposed in Eq.~\ref{eq:l_eig}.
The proof of sign invariance under the above normalization is provided in Appendix~\ref{sec:normalization}.

\section{Hopf bifurcation stability analysis}
\label{sec:analysis}
We now describe the Hopf bifurcation stability analysis and the information required for computing the stability metric. 
The Hopf bifurcation stability analysis requires information from the equilibrium point Eq.~\ref{eq:equilibrium} (solution and first-, second-, and third-order partial derivative matrices) and the solution of the right and left EVPs (Eqs.~\ref{eq:REVP} and \ref{eq:LEVP}).
The stability problem in residual form can be written as
\begin{equation}
\label{eq:stability}
\mb{r}_{\mathrm{Hopf}}(\mb{u}, \mb{x}) =
\begin{bmatrix}
\mb{r}_R \\
\mb{r}_L
\end{bmatrix},\quad
\mb{u} = \begin{bmatrix}
\mb{u}_R \\
\mb{u}_L
\end{bmatrix}.
\end{equation}
Figure~\ref{fig:schematic} shows an eXtended Design Structure Matrix (XDSM) diagram~\cite{Lambe2012a} of the process.
As mentioned above, the left and right EVPs are solved separately.
One way to solve these nonlinear equations is to use the Newton method.
However, the choice of nonlinear solver is not unique, and other methods such as nonlinear block Gauss--Seidel method may also be applicable.
The input of the left and right EVPs are the design variables, and the output are the state variables including the critical bifurcation parameters, the bifurcation eigenvectors, and the equilibrium point solution.
Some of the outputs are duplicated, and we take those from the right EVPs.
Afterwards, the first-, second-, and third-order partial derivative matrices are constructed using the bifurcation solution.
Finally, the Lyapunov coefficient is evaluated using the aforementioned matrices and the left- and right-eigenvector solutions from the left- and right-vectors, respectively.
\begin{figure}[H]
\centering
\includegraphics[angle=0,width=0.9\linewidth]{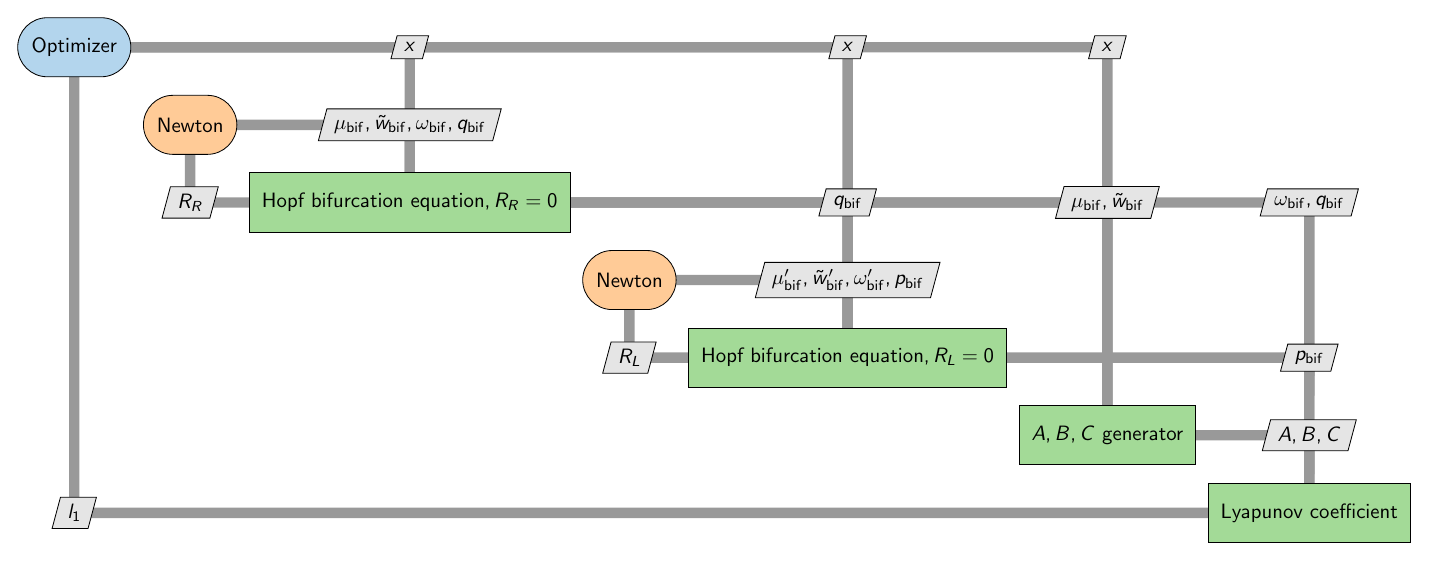}
\caption{XDSM for the bifurcation stability analysis and optimization.}
\label{fig:schematic}
\end{figure}

\section{Hopf bifurcation stability derivative computation}
\label{sec:stability_der}

To apply gradient-based optimization method to solve the optimization problem, we need to compute the derivative of functions of interest with respect to the design variables.
Specifically, we compute the derivative of the Hopf bifurcation stability parameter, which appears as a stability constraint in later optimization problems, using the coupled adjoint method~\cite[Sec.~13.3.3]{Martins2022}. 
We derive the adjoint equation for a general function of interest, $f$, and the Hopf bifurcation parameter is a special case of it.
For a general function of interest, $f$, defined as 
\begin{equation}
f = f(\mb{u},\mb{x}) .
\end{equation}
Then, we demonstrate that the coupled adjoint can be solved sequentially to reduce the computational cost.
The derivative can be computed using the following adjoint equation
\begin{equation}
\begin{aligned}
\f{\d f}{\d \mb{x}} =                                                 & \f{\p f}{\p \mb{x}} - \boldsymbol{\psi}^\intercal \f{\p \mb{r}_{\mathrm{Hopf}}}{\p \mb{x}} \\\
\f{\p \mb{r}_{\mathrm{Hopf}}}{\p\mb{u}}^\intercal \boldsymbol{\psi} = & \f{\p f}{\p \mb{u}}^\intercal.
\end{aligned}
\end{equation}

Expanding the above adjoint equation using Eq.~\ref{eq:stability}, and noting the independence of $\mb{r}_{R}$ with respect to $\mb{u}_L$ in Eq.~\ref{eq:REVP}, we have
\begin{equation}
\begin{bmatrix}
\f{\p \mb{r}_R}{\p \mb{u}_R}^\intercal & \f{\p \mb{r}_L}{\p \mb{u}_R}^\intercal \\
\mb{O}                                 & \f{\p \mb{r}_L}{\p \mb{u}_L}^\intercal
\end{bmatrix}
\begin{bmatrix}
\boldsymbol{\psi}_R \\
\boldsymbol{\psi}_L
\end{bmatrix}
=
\begin{bmatrix}
\f{\p f}{\p \mb{u}_R}^\intercal \\
\f{\p f}{\p \mb{u}_L}^\intercal
\end{bmatrix}.
\end{equation}
We can solve these two equations using block back-substitution to obtain 
\begin{equation}
\label{eq:seq_adjoint}
\begin{aligned}
\f{\p \mb{r}_L}{\p \mb{u}_L}^\intercal \boldsymbol{\psi}_L= & \f{\p f}{\p \mb{u}_L}^\intercal,                                                              \\
\f{\p \mb{r}_R}{\p \mb{u}_R}^\intercal \boldsymbol{\psi}_R= & \f{\p f}{\p \mb{u}_R}^\intercal - \f{\p \mb{r}_L}{\p \mb{u}_R}^\intercal \boldsymbol{\psi}_L.
\end{aligned}
\end{equation}
Block back-substitution decomposes the original coupled system of linear equations into two half-size decoupled subsystems of linear equations.
This is beneficial because each subsystem is smaller and cheaper to solve, and the approach can leverage existing code and solvers for each subsystem.
More specifically, if there is already a bifurcation adjoint solver, then the bifurcation stability adjoint computation can directly call the existing solution, significantly reducing the overall development effort.
For more details about the two adjoint EVP equations, see Appendices~\ref{sec:adjoint_L} and \ref{sec:adjoint_R}.

The total derivative can be written as
\begin{equation}
\label{eq:tot_der}
\f{\d f}{\d \mb{x}} = \f{\p f}{\p \mb{x}} - \boldsymbol{\psi}_R^\intercal \f{\p \mb{r}_R}{\p \mb{x}} - \boldsymbol{\psi}_L^\intercal \f{\p \mb{r}_L}{\p \mb{x}}.
\end{equation}
In contrast to many explicitly defined functions of interest, the first Lyapunov coefficient Eq.~\ref{eq:lyp} is an implicit function, meaning that to evaluate the function value, additional equations need to be solved.
This is because the matrix inversion in Eq.~\ref{eq:lyp} is typically implicitly evaluated via a linear solution.
In addition, the formula also involves Jacobian, Hessian, and third-order partial derivative matrices (tensors), {which must be differentiated with respect to the design variables.}

We have derived an analytic RAD formula leveraging the complex analytic function dot-product identity tool developed by \citet{He2023}, with details presented in Appendix~\ref{sec:partial_l}.
The RAD formulation of $f_{\text{Lyp}}$ contains two adjoint equations.
Thus, to compute the partial derivatives of the first Lyapunov coefficient, we need to solve two $(3n+3) \times (3n+3)$ and two $(n\times n)$ linear system problems, where $n$ is the state dimension.
However, this computational cost does not depend on the number of design variables and is thus {applicable} to large-scale optimization problems.

\section{Hopf bifurcation stability-constrained optimization problem formulation}
\label{sec:opt}

{The first Lyapunov coefficient plays a crucial role in our methodology by enabling us to determine the stability of Hopf bifurcations using only local information about the system near the bifurcation point. 
Unlike methods that require computing the entire branch of limit cycle oscillations, the first Lyapunov coefficient allows us to characterize the bifurcation as subcritical or supercritical based solely on the right and left eigenvectors and the system's Jacobian, Hessian, and third-order derivatives evaluated at the bifurcation point itself. 
This makes the first Lyapunov coefficient a highly efficient metric for stability analysis and an ideal candidate for integration into gradient-based optimization frameworks.}

{A general} Hopf bifurcation stability-constrained optimization problem is formulated as follows
\begin{equation}
\label{eq:general_optimization_statement}
\begin{aligned}
\min \quad        & f(\tilde{\mb{w}}(\mb{x}), \mu(\mb{x}), \mb{x})           \\
\text{by varying} \quad & \mb{x}                                                   \\
\text{subject to} \quad  & f_{\mathrm{lyp}}(\mb{u}_R(\mb{x}), \mb{u}_L(\mb{x}),\mb{x})\leq\bar{f}_{\mathrm{lyp}},\\
\quad & g_i(\tilde{\mb{w}}(\mb{x}), \mu(\mb{x}), \mb{x}) \leq 0
\end{aligned},
\end{equation}
where $f$ is the objective function, {$f_{\mathrm{lyp}}$ is the bifurcation stability constraint and, $g_i$ are additional inequality constraints}.
For example, {$f$ could be a design cost metric and $g_i$ could be constraints on the upper and lower bounds of the design variables or bifurcation parameter}.
The evaluation of $f_{\mathrm{lyp}}$ implicitly enforces Eq.~\ref{eq:stability} and $\bflyp<0$ is a user-specified upper bound of the first Lyapunov coefficient.
By enforcing Eq.~\ref{eq:stability}, we ensure that the solution tracks the bifurcation point.

This problem formulation can represent various aerospace applications, {such as minimizing the design cost while enforcing an additional constraint to ensure the onset flutter velocity is above a certain threshold}, where the bifurcation parameter corresponds to the flutter speed.
{In this case, we have $g_i = \mu(\mb{x})$}. 
{Different from an alternative formulation that fix flight condition (including $\mu$) and enforce the stability by enforcing $\text{max}_i\text{Re}(\lambda_i)<0$, our formulation traces the onset of bifurcation. 
It means we will always have eigenvalues that on the imaginary axis, $\text{Re}(\lambda)=0$.}

\section{Numerical results}
\label{sec:example}

In this section, we verify and demonstrate the proposed constraint and its derivative computation.
{Three} optimization results are presented.
In Section~\ref{sec:algebraic_problem} we apply the proposed method to an algebraic Hopf bifurcation model, where we also present derivative verification.
In Section~\ref{sec:aeroelastic_problem} we solve a nonlinear aeroelastic design problem, {and in Section}~\ref{sec:CGL_example} {we solve a nonlinear PDE problem.}

\subsection{Algebraic model problem}
\label{sec:algebraic_problem}

\subsubsection{Problem definition}
We consider the following dynamical system,
\begin{equation}
\label{eq:algebraic_system}
\begin{aligned}
\dot{w}_1 = & (\mu - x_1) w_1 - w_2 + (2 x_1 x_2 - 1 ) w_1^3, \\
\dot{w}_2 = & w_1 + (\mu - x_2) w_2 + (2 x_2 - 1 ) w_2^3.
\end{aligned}
\end{equation}
where the dynamical system is parametrized by the design variable vector $\mb{x} = \left[x_1, x_2 \right]^\intercal$.
By varying the design vector, $\mb{x}$, the bifurcation solution can be subcritical or supercritical.
{This is a manufactured problem with its primary purpose being to verify derivatives and demonstrate the proposed method in an optimization.}

\subsubsection{Derivative verification}

First, we verify that the adjoint algorithm and AD formulae can be used to compute the derivative accurately.
We choose an arbitrary design point, $\mb{x} = \left[ 0.2, 0.7 \right]^\intercal$.
The total derivative computed using the adjoint method (Eq.~\ref{eq:seq_adjoint}) is shown in Table~\ref{tab:partial_l}.
We achieve a relative error of approximately $10^{-5}$ compared to the finite difference method, verifying that the adjoint formula computes the derivative accurately.
{We use the finite difference method to verify derivatives, rather than the more accurate complex step method}~\cite[Sec.~6.5]{Martins2022} {because the Newton solvers used in the equilibrium point solution as well as the left and right EVPs are not complex step safe.}

We also compute the partial derivatives of the first Lyapunov coefficient with respect to the left and right EVP solution vectors $\mb{u}_L, \mb{u}_R$ ($\omega$ and $\mu$ are entries from $\mb{u}_R$) calculated using Eq.~\ref{eq:stability}.
The final results are based on a hybrid application of analytically derived RAD formulae (see Eq.~\ref{eq:lyp}) and a generic RAD implementation using JAX~\cite{jax2018github}.
{More specifically, we use JAX to differentiate the lower level functions from Appendix} \ref{sec:partial_l} {($\mb{A}, \mb{b}, \mb{c}$) with respect to the multiple inputs using vector-Jacobian products, and we differentiate the upper level function using the analysis presented in Section}~\ref{sec:stability_der}. 
The results are listed in Table~\ref{tab:algebraic_example_mat}.
The finite difference and AD results differ by $10^{-7}$ to $10^{-11}$, which verifies the accuracy of the developed derivative formulae.
When the variable is a vector, we arbitrarily sampled one entry from the vector, and the sampled index is denoted by $[i]$.

\begin{table}[h]
\centering
\caption{Comparison of the total stability derivative computed using the adjoint method and finite differences.}
\begin{tabular}{crrr}
\toprule
$\mb{x}$ & Finite difference & Adjoint                  & Relative difference   \\
\midrule
$\f{\d l}{\d x_1}$    & $1.11520755$      & $1.115\underline{16669}$ & $3.66 \times 10^{-5}$ \\
$\f{\d l}{\d x_2}$    & $1.80103018$      & $1.80\underline{092820}$ & $5.66 \times 10^{-5}$ \\
\bottomrule
\end{tabular}
\label{tab:partial_l}
\end{table}

\begin{table}[h]
\centering
\caption{Comparison of the stability partial derivatives computed using AD and finite differences.}
\begin{tabular}{crrr}
\toprule
$\mb{x}$                    & Finite difference & AD                        & Relative difference         \\
\midrule
$\f{\p l}{\p \mb{x}[1]}$     & $1.10061847$      & $1.1006184\underline{8}$  & $4.03139078\times 10^{-9}$  \\
$\f{\p l}{\p \mb{u}_{L}[3]}$ & $-0.18662431$     & $-0.18662431$             & $5.61530267\times 10^{-11}$ \\
$\f{\p l}{\p \mb{u}_{R}[1]}$ & $-0.03551026$     & $-0.0355\underline{0959}$ & $6.65218077\times 10^{-7}$  \\
$\f{\p l}{\p \omega}$       & $0.25320698$      & $0.253206\underline{73}$  & $2.53628802\times 10^{-7}$  \\
$\f{\p l}{\p \mu}$          & $-0.02632296$     & $-0.02632296$             & $1.84866444\times 10^{-9}$  \\
\bottomrule
\end{tabular}
\label{tab:algebraic_example_mat}
\end{table}

\subsubsection{Optimization}

Having verified the derivatives, we can now perform gradient-based optimization.
The optimization problem is as defined in Eq.~\ref{eq:general_optimization_statement} with no additional constraints.
The objective of the optimization is to maximize the bifurcation parameter, which is equivalent to minimizing $f=-\mu$.
{In engineering design applications, it may be desired to minimize a design metric while constraining the lower bound of the bifurcation parameter. For example minimizing the fuel burn in an aircraft while enforcing the minimum flutter onset velocity via a constraint.
In this example, we simply maximize the bifurcation parameter while ensuring its stability with no design metric considerations.}
We use IPOPT~\cite{Wachter2006a}, an interior point optimizer, using the Python wrapper provided by pyOptSparse~\cite{Wu2020a} to solve this optimization problem.

Design space contours of both the first Lyapunov coefficient and the function of interest, together with the path taken by the optimizer, are shown in Fig.~\ref{fig:swept}.
\begin{figure}[ht]
\centering
\includegraphics[angle=0,width=0.99\linewidth]{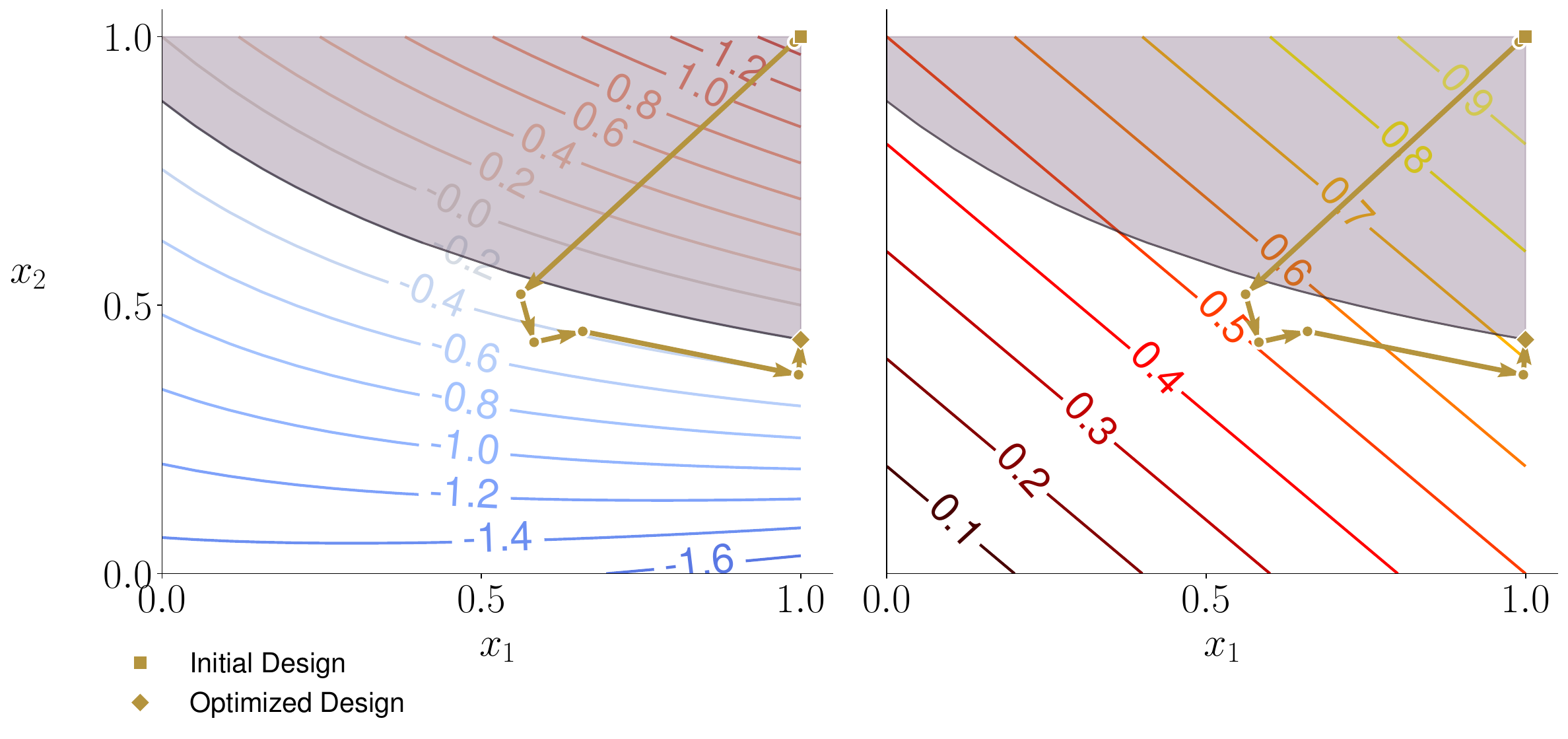}
\caption{Contours of the stability metric, $f_{\text{Lyp}}$ (left), and the bifurcation parameter, $\mu$, (right).
The brown arrows indicate the path of the optimization.}
\label{fig:swept}
\end{figure}
In this case, we set the upper bound for the first Lyapunov coefficient to $\bflyp=-0.2$.
Starting from initial design at $\mb{x}=\left[1, 1\right]$, the optimized solution $\mb{x}^* = \left[1, 0.4359\right]$ is found in seven iterations.
The LCO responses of all major iterations are shown in Fig.~\ref{fig:history}.
The initial and the optimized solution LCO curves and responses to small perturbations are plotted in Fig.~\ref{fig:base_opt}.
For the initial design, the LCO is subcritical, and small perturbations inside or outside the LCO cause the trajectory to either diverge or converge to the equilibrium point solution.
In contrast, the final solution is a supercritical LCO, and the perturbed solution converges to the stable LCO. 
Although the objective function decreases from about 1.0 to 0.73, the undesirable subcritical bifurcation disappears. 

\begin{figure}[ht]
\centering
\includegraphics[angle=0,width=0.7\linewidth]{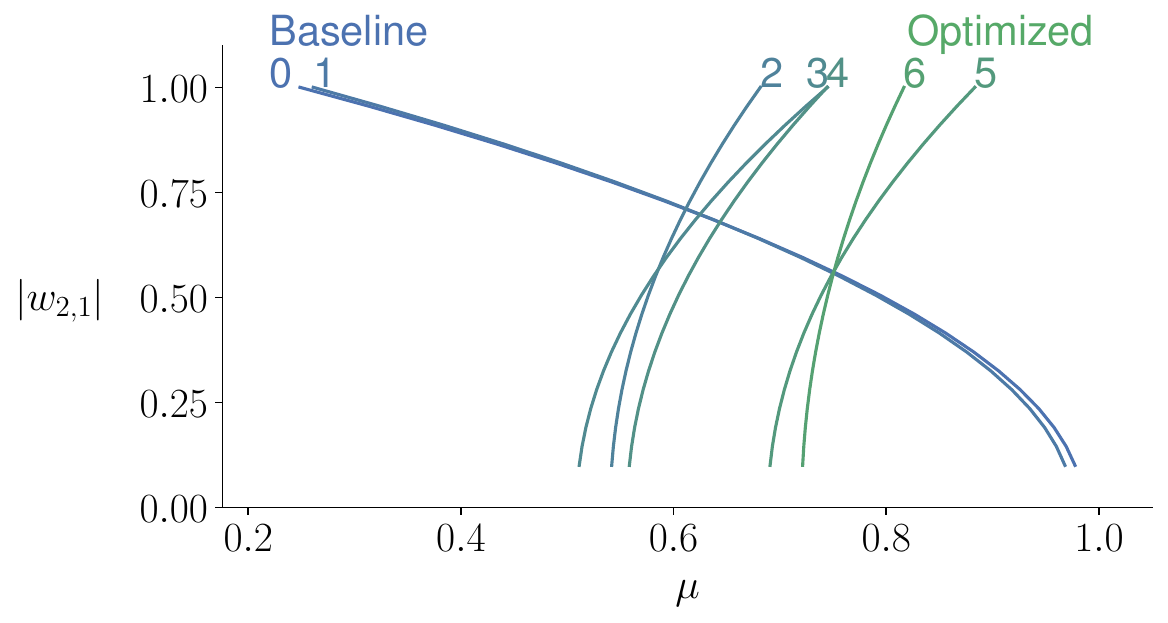}
\caption{Optimization history by major iteration showing response going from subcritical to supercritical. The color indicates the iteration index.}
\label{fig:history}
\end{figure}

\begin{figure}[ht!]
\centering
\includegraphics[angle=0,width=0.99\linewidth]{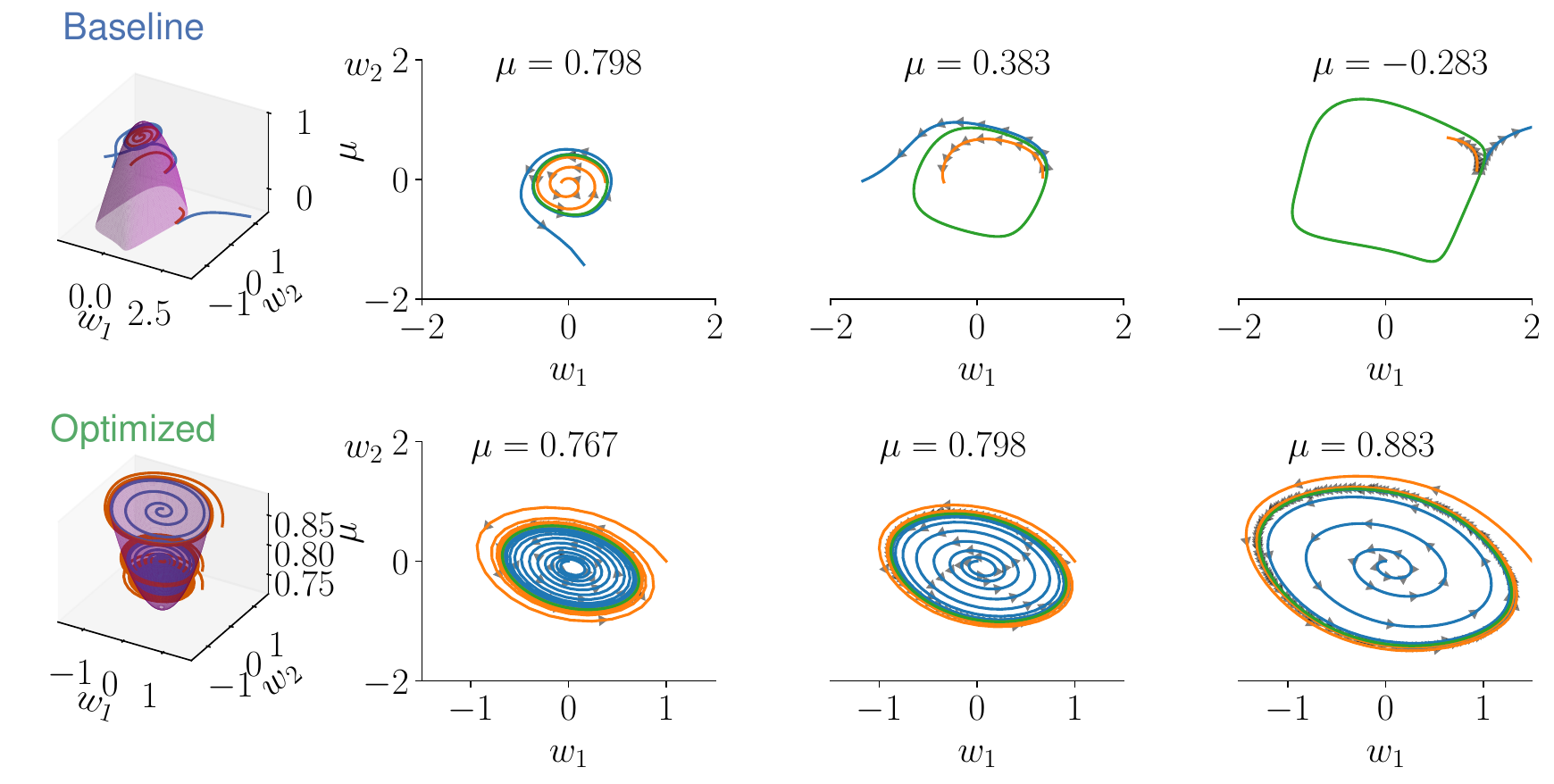}
\caption{Baseline (top) and optimized (bottom) response for the algebraic example (Eq.~\ref{eq:algebraic_system}). {The optimized design eliminates the subcritical bifurcation and corresponding LCOs seen in the baseline response.}}
\label{fig:base_opt}
\end{figure}

\subsection{Aeroelastic problem}
\label{sec:aeroelastic_problem}

\subsubsection{Problem definition}
In this section, we examine a nonlinear aeroelastic optimization problem using a model adapted from \citet{Riso2023a}.
The aerodynamic model represents a typical section within a quasi-steady incompressible potential flow, as illustrated in Fig.~\ref{fig:AirfoilDiagram}.
The model accounts for the effects of the effective angle of attack, apparent mass, and inertia, while neglecting wake effects.
The airfoil is constrained in both translation and rotation by springs, with nonlinear torsional stiffness.
\begin{figure}[htb!]
\centering
\includegraphics[width=0.6\linewidth]{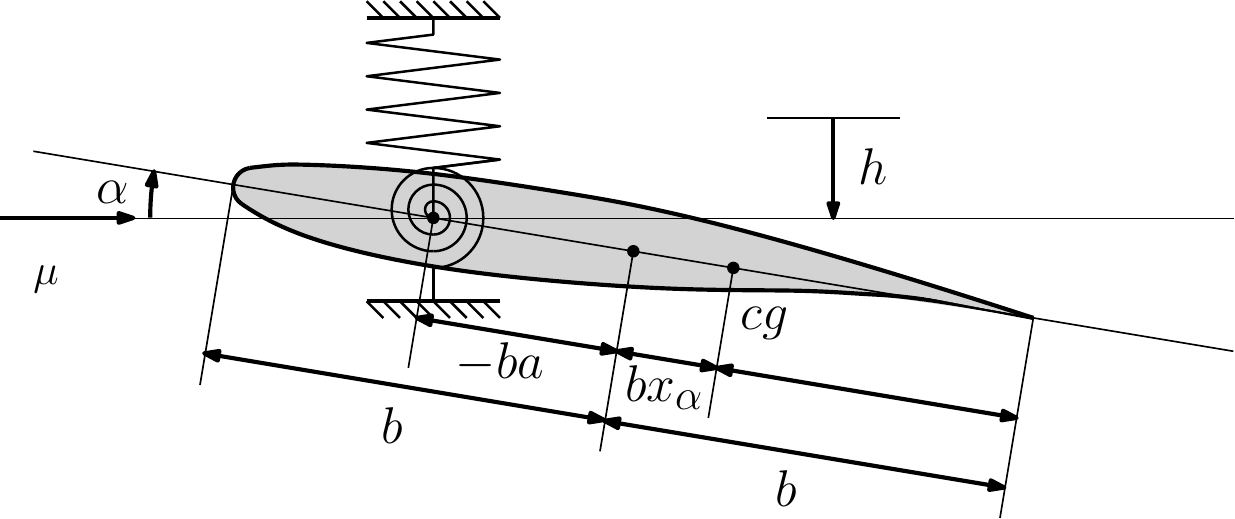}
\caption{Schematic of typical section.}
\label{fig:AirfoilDiagram}
\end{figure}

The dynamical system is represented as follows:
\begin{equation}
\label{eq:algebraic}
\dot{\mb{w}} = \mb{f}(\mb{w}, \mb{x}) = \mb{A}(\mb{x})\mb{w} + \mb{F}_{nl}(\mb{w}, \mb{x}),
\end{equation}
where the dynamics, $\mb{f}(\mb{w}, \mb{x})$, is composed of a linear coefficient matrix, $\mb{A}(\mb{x})$, and a state-dependent forcing term, $\mb{F}_{nl}(\mb{w}, \mb{x})$.
The state variable vector, $\mb{w}$, is defined as
\begin{equation}
\mb{w} =
\begin{bmatrix}
\overbar{h},       &
\alpha,            &
\dot{\overbar{h}}, &
\dot{\alpha}
\end{bmatrix}^\intercal,
\end{equation}
where $\overbar{h}$ and $\alpha$ are the dimensionless plunging and pitching variables, respectively.
The design variable vector, $\mb{x}$, is defined as
\begin{equation}
\mb{x} =
\begin{bmatrix}
\overbar{m}, &
\kappa_\alpha^{(3)}
\end{bmatrix}^\intercal.
\end{equation}
Here, \( \overbar{m} \) is a dimensionless mass per unit length, defined as the ratio of the mass per unit length of the typical section to the mass of the fluid within a circle of radius equal to the half chord.
The parameter \( \kappa_\alpha^{(3)} \) is a dimensionless stiffness coefficient, representing the ratio between the third-order and first-order rotational spring stiffness constants.
In the optimization process, \( \overbar{m} \) is adjusted by varying the mass per unit length, while the fluid mass in the circular region with a half-chord radius remains constant.

The linear coefficient matrix is defined as
\begin{equation}
\mb{A}(\mb{x}) =
\begin{bmatrix}
\mb{0}_{2\times 2}  & \mb{I}_{2\times 2}   \\
-\mb{M}^{-1} \mb{K} & -\mb{M}^{-1}\mb{D}_A
\end{bmatrix}.
\end{equation}
The individual aerodynamic and structural mass, stiffness, and damping matrices are defined as
\begin{equation}
\mb{M}(\mb{x})
= \mb{M}_S + \mb{M}_A(\mb{x}), \quad
\mb{M}_S =
\begin{bmatrix}
1        & x_\alpha   \\
x_\alpha & r_\alpha^2
\end{bmatrix}, \quad
\mb{M}_A(\mb{x}) = \f{1}{\overbar{m}}
\begin{bmatrix}
1  & -a                          \\
-a & \left(\f{1}{8} + a^2\right)
\end{bmatrix}, \quad
\end{equation}
\begin{equation}
\mb{K}(\mb{x})
= \mb{K}_S + \mb{K}_A(\mb{x}), \quad
\mb{K}_S =
\begin{bmatrix}
\Omega^2 & 0          \\
0        & r_\alpha^2
\end{bmatrix},\quad
\mb{K}_A(\mb{x})
=\frac{2\mu^2}{\overbar{m}}
\begin{bmatrix}
0 & 1                              \\
0 & - \left(\frac{1}{2} + a\right)
\end{bmatrix},
\end{equation}
\begin{equation}
\mb{D}_A(\mb{x}) = \frac{2\mu}{\overbar{m}}
\begin{bmatrix}
1                          & 1-a                         \\
-\left(\f{1}{2} + a\right) & a \left(a - \f{1}{2}\right)
\end{bmatrix},
\end{equation}
where $a, x_\alpha$, $\Omega$, $r_\alpha$, and $\mu$ are non-dimensional parameters that are fixed in the optimization.
The nonlinear load is defined as
\begin{equation}
\mb{F}_{nl}(\mb{w}, \mb{x})
=
\begin{bmatrix}
0 \\
0 \\
-\mb{M}^{-1}\begin{bmatrix}
0 \\
r_\alpha^2\left(\kappa_\alpha^{(3)}\alpha^3 + \kappa_{\alpha}^{(5)}\alpha^5\right)
\end{bmatrix}
\end{bmatrix}
\end{equation}
where $\kappa_\alpha^{(5)}$ is a dimensionless stiffness coefficient, defined as the ratio of the fifth-order and the first-order spring rotational stiffness constants.
The values of all constants are listed in Table~\ref{tab:const_param}.
{The constants used are exactly the same as those used in the optimization problem solved by} \citet{Riso2023a}.
{At the equilibrium point, the nonlinear forcing term becomes a zero vector with $\mb{\tilde{w}} = \mb{0}$.
This greatly simplifies both the forward and backward pass dynamics.
However, we solve a PDE-governed problem with a non-zero equilibrium solution in Section}~\ref{sec:CGL_example}.

\begin{table}[!h]
\caption{Dimensionless parameters for the aeroelastic optimization problem.}
\begin{center}
\ra{1.2}
\begin{tabular}{@{}llr@{}}
\toprule
\textbf{Property}     & \textbf{Description}                                   & \textbf{Value} \\
\midrule
$a$                   & Offset of the elastic center from the half-chord point & $-0.3$         \\
$\Omega$              & Plunge-pitch structural frequency ratio                & $0.5$          \\
$r_\alpha$            & Radius of gyration                                     & $0.3$          \\
$x_\alpha$            & Static unbalance                                       & $0.2$          \\
$\kappa_\alpha^{(5)}$ & Fifth-order stiffness coefficient                      & $100.0$         \\
$\mu$                 & Speed (bifurcation parameter)                          & $0.8$          \\
\bottomrule
\end{tabular}
\end{center}
\label{tab:const_param}
\end{table}

\subsubsection{Optimization}
The aeroelastic optimization problem statement is summarized as follows,
\begin{equation}
\label{eq:optimization}
\begin{aligned}
\min \quad         & \overbar{m} - {\kappa_{\alpha}^{(3)}}^2                                                                     \\
\text{by varying} \quad & \underline{\mb{x}} \leq \begin{bmatrix}\overbar{m}\\ {\kappa_{\alpha}^{(3)}}\end{bmatrix} \leq \bar{\mb{x}} \\
\text{subject to} \quad  & \flyp \leq \bflyp                                                                                               \\
\quad              & \mu \geq\underline{\mu}                                                                                     \\
\end{aligned} .
\end{equation}
The objective function consists of the dimensionless mass and the cubic stiffness coefficient squared, which are also the design variables. 
The dimensionless mass is varied by changing the mass per unit length while keeping the mass of the fluid in a circle with the half chord as the radius fixed.
The stiffness coefficient squared is subtracted to penalize the nonlinearity of the design.

As before, we minimize the objective function, subject to the equilibrium point stability constraint.
The stability bound is set to $\bflyp = -0.1$
{to ensure that the bifurcation is supercritical.
We choose a number sufficiently below zero to avoid a marginally stable design solution.}
{We also constrain the speed, $\mu$, namely the bifurcation parameter, to be greater than} $\underline{\mu}$ to avoid low-speed flutter onset.
{We set the lower bound, $\underline{\mu}=1.01$ to be exactly equal to the bifurcation parameter of the baseline design in this study.
Therefore, the optimization problem reduces to identifying an optimal design that maximizes the objective function while ensuring that the flutter onset speed does not decrease and that the resulting bifurcation remain supercritical.}
No other constraints are enforced in this optimization.
Finally, the bounds on the design variables are
\begin{equation}
\underline{\mb{x}} =
\begin{bmatrix}
5 \\
-3
\end{bmatrix},
\quad
\bar{\mb{x}} =
\begin{bmatrix}
17 \\
1
\end{bmatrix}.
\end{equation}
These bounds are chosen large enough to allow the optimizer to stiffen or weaken the structure as needed, affecting nonlinearity, subject to the stability constraint.
The initial design variable values are arbitrarily chosen as
$
\mb{x} = \left[15.00, -3.00\right]^\intercal.
$

As before, we use IPOPT as the optimizer through the pyOptSparse~\cite{Wu2020a} wrapper for this study.
The contour plots of the objective function, the flutter speed constraint, and the bifurcation stability constraint, together with the path taken by the optimizer, are shown in  Fig.~\ref{fig:swept_ae}.
In {four} iterations, the optimized solution is found to be
$
\mb{x}^* = \left[15.00, 1.00\right]^\intercal,
$
where the flutter speed constraint is active, and the bifurcation stability constraint is satisfied but not active.
\begin{figure}[ht]
\centering
\includegraphics[angle=0,width=0.99\linewidth]{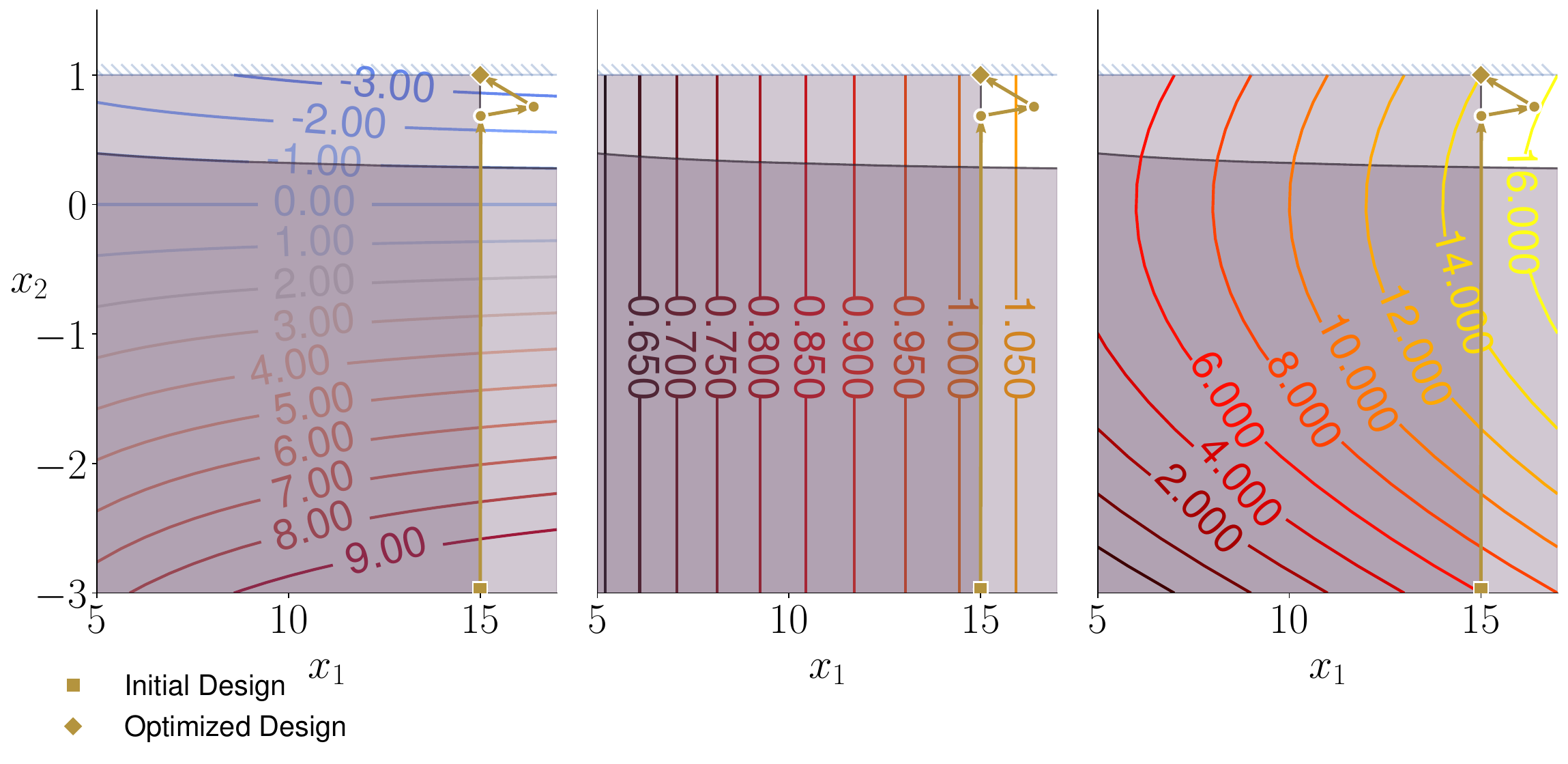}
\caption{Contours of the stability metric, $f_{\text{Lyp}}$ (left), the bifurcation parameter, $\mu$, (middle), and the objective function (right).
The brown arrows indicate the path of the optimization.}
\label{fig:swept_ae}
\end{figure}
This is in contrast to the initial design, $\mb{x}^{(0)}$, where the flutter speed constraint is active, and the bifurcation stability constraint is violated.
The distinction between the baseline and optimized designs becomes more evident by examining the LCO curves for all major iterations shown in Fig.~\ref{fig:history_ae}.
The baseline design exhibits subcritical LCO behavior near the equilibrium solution, whereas the optimized design demonstrates supercritical LCO behavior.

\begin{figure}[ht!]
\centering
\includegraphics[angle=0,width=0.7\linewidth]{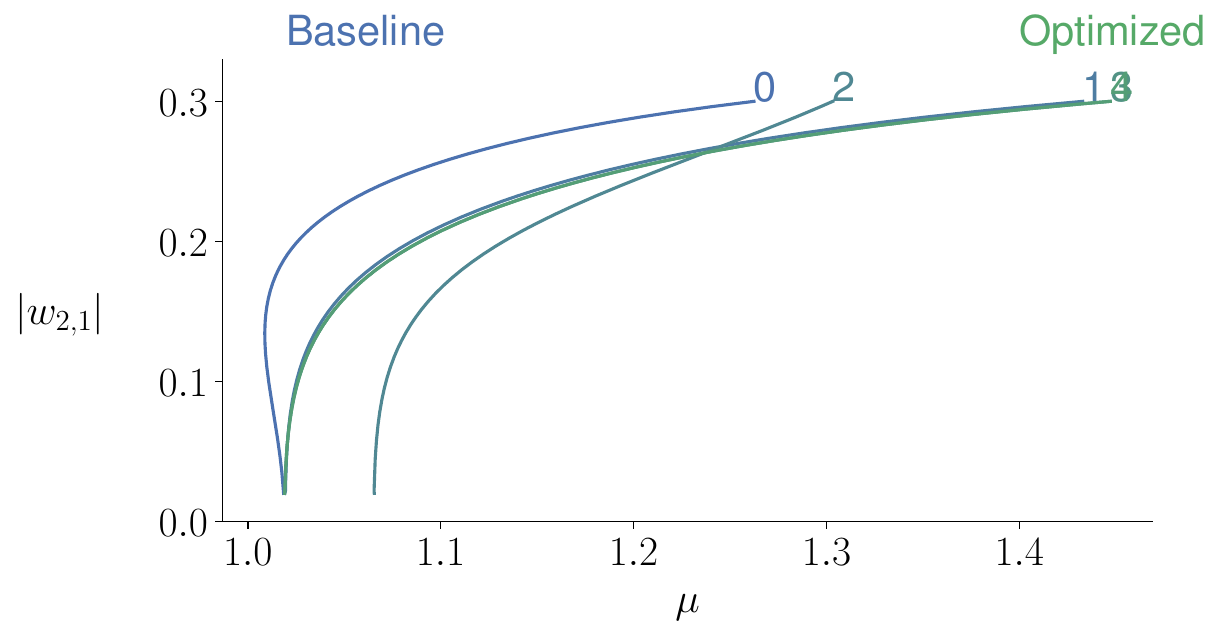}
\caption{Optimization history of iterations showing the LCO curve going from subcritical to supercritical.}
\label{fig:history_ae}
\end{figure}

Taking several prescribed motion magnitudes on the baseline and the optimized solution and adding a small perturbation, we obtain the time-accurate simulation of the trajectories shown in Fig.~\ref{fig:base_opt_ae}. 
{When simulated with a lower flutter speed value}, the baseline design has {subcritical Hopf-bifurcation: the inner LCO is unstable, while the outer LCO is stable.
Starting from a solution in the vicinity of the inner unstable LCO, the solution either converges to the equilibrium point or the external stable LCO.
As we simulate the system with higher flutter speed values, shown from left to right plots, the two LCOs merge into a single stable LCO.}
{For the optimized design, all perturbed solutions converge to a single supercritical LCO, indicating that the LCO is stable.}

To summarize, the optimization successfully {finds the optimal design point that does not include any subcritical bifurcation while also ensuring that the flutter speed does not decrease from the baseline design.}
Starting from a design {that exhibits unstable subcritical LCOs, we discovered an optimal design with better performance.}

\begin{figure}[ht]
\centering
\includegraphics[angle=0,width=0.99\linewidth]{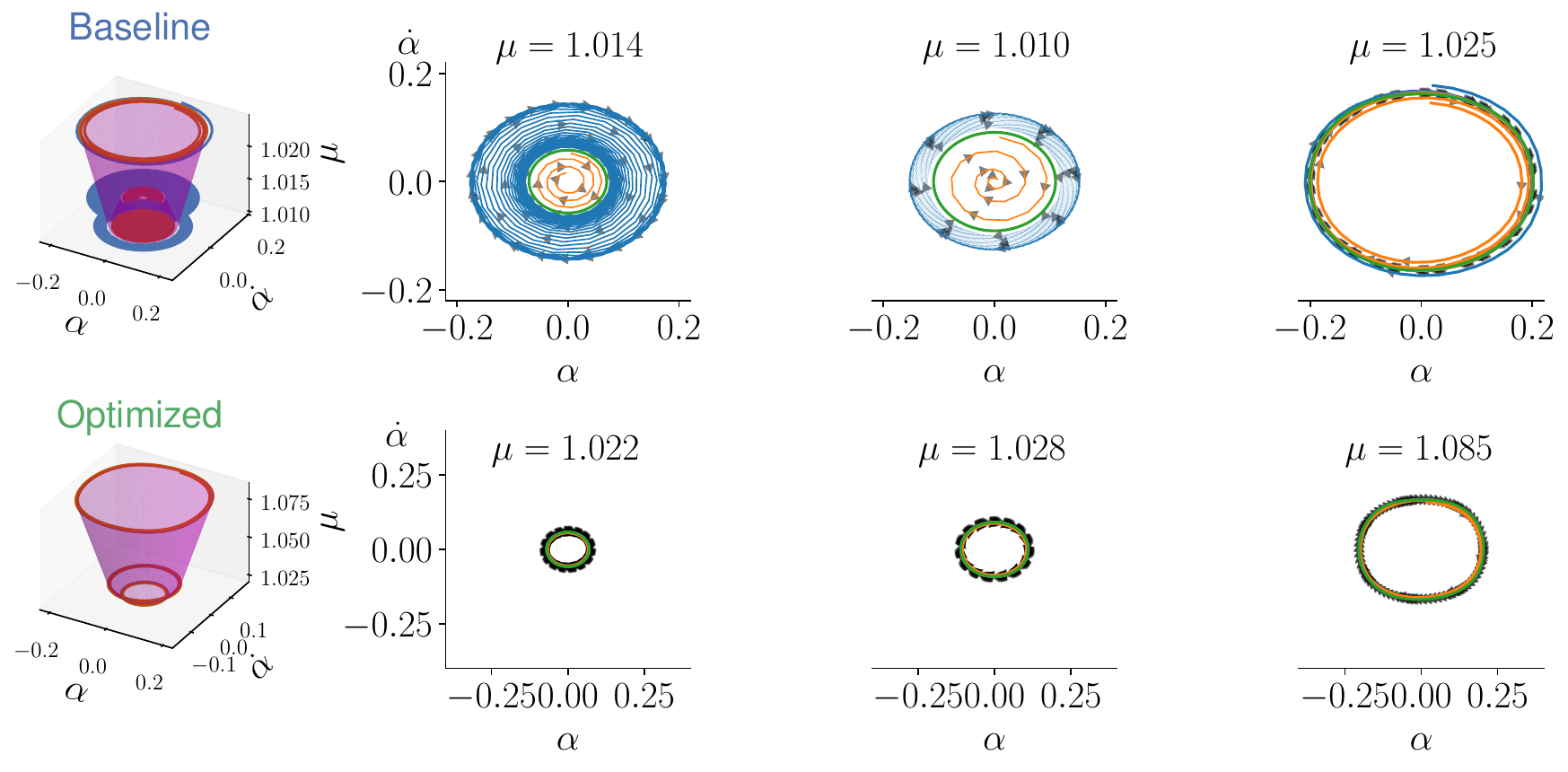}
\caption{Baseline (top) and optimized (bottom) response for the aeroelastic example (Eq.~\ref{eq:optimization}). {The baseline design includes subcritical bifurcation, while the optimized design has only stable supercritical bifurcation.}}
\label{fig:base_opt_ae}
\end{figure}

\subsection{{Complex Ginzburg--Landau Problem}}
\label{sec:CGL_example}

\subsubsection{Problem Formation}

{To demonstrate the ability of the proposed method to solve large dimensional nonlinear problems with moving steady state solutions, we consider the complex Ginzburg--Landau equation.
This equation is one of the most well studied PDEs in physics and describes the behavior of many different pattern forming phenomena such as superconductivity, Rayleigh--Bernard convection, and reaction-diffusion systems.
A detailed overview of the dynamics and bifurcation structure of the equation is given by} \citet{Aranson2002}.
{One form of the governing equation is}
\begin{equation}
\f{\partial{w}}{\partial t} = \Delta{w} + (\mu+j\nu){w} - (c_3(\xi)+j\sigma)|{w}|^2{w} - c_5(\xi)|{w}|^4{w} + f(\xi),
\end{equation}
{with homogeneous Neumann boundary conditions.
Here, $\xi$ is the coordinate, ${w}$ is the state variable, $c_3(\xi)$ and $c_5(\xi)$ are field parameters (later we treat $c_3$ as a design variable and $c_5$ is a fixed parameter), $\mu$ is the bifurcation parameter, the Laplacian operation is denoted by $\Delta$, and $\nu$ and $\sigma$ are constants.
A nonlinear source term $f(\xi)$ is added to the system.
We consider the 1-D domain $-\pi \leq \xi \leq \pi$ with the dimension of the system $n=32$.
The equation can be split into its real and imaginary parts to make a two-component system,}
\begin{equation}
\f{\partial}{\partial t}\begin{bmatrix}{w}_1\\ {w}_2\end{bmatrix}
=
\begin{bmatrix}
\Delta+\mu & -\nu \\
\nu & \Delta +\mu
\end{bmatrix}
\begin{bmatrix}
{w}_1 \\
{w}_2
\end{bmatrix}
- ({w}_1^2 + {w}_2^2)
\begin{bmatrix}
c_3{w_1} - \sigma{w}_2\\
\sigma {w}_1 + c_3{w}_2
\end{bmatrix}
- c_5({w}_1^2 + {w}_2^2)^2
\begin{bmatrix}
{w}_1\\
{w}_2
\end{bmatrix}
+
\begin{bmatrix}
f_{\text{real}}(\xi)\\
f_{\text{im}}(\xi)
\end{bmatrix}
,
\end{equation}
{where ${w}_1$ represents the real part of the state variable and ${w}_2$ represents the imaginary part.
The equation is solved using a finite difference approach to calculate the Laplacian of the state variable.

The design variable for this problem is $c_3(\xi)$ and all other variables are defined as}
\begin{equation}
\begin{aligned}
c_5(\xi) &= \text{tanh}(\xi),\\
f_{\text{real}}(\xi) &= \f{1}{2}e^{\f{-\xi^2}{2}},\\
f_{\text{im}} &= 0,\\
\nu &= 1,\\
\sigma &= 0.1,
\end{aligned}
\end{equation}
{and are fixed throughout the optimization.
The nonlinear source term is Gaussian and applied only to the real part of the system.}

\subsubsection{Optimization}

{The objective function considered for this optimization is defined as}
\begin{equation}
f = f_{\text{min}} + f_{\text{smooth}} + f_{\text{var}}, 
\end{equation}
where  
\begin{equation}
\begin{aligned}
f_{\text{min}} =& \int c_3(\xi)\d\xi, \\ 
f_{\text{smooth}} =& \int \left(\f{d c_3}{d\xi}(\xi)\right)^2\d\xi, \\ 
f_{\text{var}} =&
- \left(\text{Var}(c_3(\xi)\right).
\end{aligned}
\end{equation}
{Here, $\text{Var}(c_3(\xi))$ is the variance of $c_3$ over $\xi$.
This objective function consists of three parts: $f_{\text{min}}$ minimizes the values of $c_3$ in the domain, $f_{\text{smooth}}$ promotes a smooth solution of the design variables, and $f_{\text{var}}$ promotes variance of the design variables in the domain.
An interpretation is to let $c_3(\xi)$ denote a spatially varying model-correction or actuation field along a surface.
In field inversion and machine learning (FIML), $c_3$ can be a local correction coefficient multiplying a turbulence-model term} \citep{Parish2016a,Duraisamy2019};
{the objective balances small corrections on average, smoothness to avoid spurious oscillations, and sufficient variance to capture genuine inhomogeneities.
Other links to instabilities in oscillatory and inhomogeneous media can be drawn from the Kuramoto framework }\citep{Kuramoto1984,Acebrn2005}.

{The optimization problem statement is summarized as follows}
\begin{equation}
\label{eq:optimization_CGL}
\begin{aligned}
\min \quad & f_{\text{min}}+f_{\text{smooth}} + f_{\text{var}}\\
\text{by varying} \quad & \underline{\mb{x}} \leq c_3(\xi)\leq \bar{\mb{x}} \\
\text{subject to} \quad  & \flyp \leq \bflyp\\
\end{aligned},
\end{equation}
{where the design variables $\mb{x} = c_3(\xi)$ have dimension of $n$.
The upper and lower bounds for the design variables are set to be constant over the domain $\xi$},
\begin{equation}
\underline{\mb{x}} =
0\quad
\overbar{\mb{x}} =
10,
\end{equation}
{and the stability constraint as $\bflyp=-0.5$ to ensure no subcritical bifurcations are present in the solution.
The initial design is given by an arbitrary selected function}
\begin{equation}
\begin{aligned}
c_3^{(0)} &= -\text{tan}(\xi).\\
\end{aligned}
\end{equation}
{The optimization problem can then be interpreted as finding the minimum concentration of a catalyst $c_3$ which ensures no subcritical bifurcation occur in the system. The smoothness term in the objective function ensures a physically realizable design and the upper and lower bounds on the design variable represent physical limitations in the system.

One challenge when considering problems governed by PDE equations is that the higher order derivatives given in Eq.}~\ref{eq:bc}{ are not available analytically and numerical approaches must be used.
This can be computationally expensive; however, we use the approach outlined by} \citet{Beran1999} {to calculate the directional derivatives as}
\begin{equation}
\begin{aligned}
\mb{b}(\mb{y}_1,\mb{y}_2,\tilde{\mb{w}},\mu,\mb{x}) &= \f{1}{2\epsilon_b}\left(\mb{A}(\tilde{\mb{w}} + \epsilon_b\mb{y}_1, \mu, \mb{x}) - \mb{A}(\tilde{\mb{w}} - \epsilon_b\mb{y}_1, \mu, \mb{x})\right)\mb{y_2}.
\end{aligned}
\end{equation}
{We use the corresponding second-order accurate finite-difference formula to calculate the vector $\mb{c}$.
The finite difference perturbations $\epsilon_b=10^{-4}$ and corresponding $\epsilon_c=0.01$ are arrived at empirically.
This is significantly less expensive than traditional finite-differences to calculate the Hessian and third order tensor as it only requires six total calls to the Jacobian matrix to calculate both terms.

Once again using IPOPT through pyOptSparse}~\cite{Wu2020a}{, the optimal solution is found in 76 iterations.
The first nine iterations as well as the final design are shown in Fig.}~\ref{fig:opt_cgl}{.
As can be seen, the baseline design included subcritical bifurcations at $\mu = 0.22$ and the optimized design only contains supercritical bifurcations at $\mu=0.46$.}
{The solution history of the design cost, design variable, and equilibrium solution is shown in Fig.}~\ref{fig:opthist_cgl}.
{As can be seen, the optimal design variable solution is smooth over the domain as promoted by $f_{\text{smooth}}$.
The equilibrium solution is also dependent on the design variable and non-zero, as opposed to the previous example.}
{Figure}~\ref{fig:cgl_ex} {shows the time-accurate simulations of the perturbed trajectories for different prescribed motion magnitudes.
The baseline design, when starting from a solution close to the equilibrium solution, exhibits subcritical bifurcation, leading to the trajectory diverging to an unstable LCO.
In contrast, starting from the same points, the optimized design exhibits only supercritical bifurcations leading the system to converge to the stable LCO.

To summarize the optimization for the CGL example, given a sub-optimal initial design that is unstable, we simultaneously minimize the design cost of the system while ensuring that the resulting optimal design exhibits only supercritical bifurcations that lead to stable LCOs.}

\begin{figure}[ht]
\centering
\includegraphics[width=0.7\linewidth]{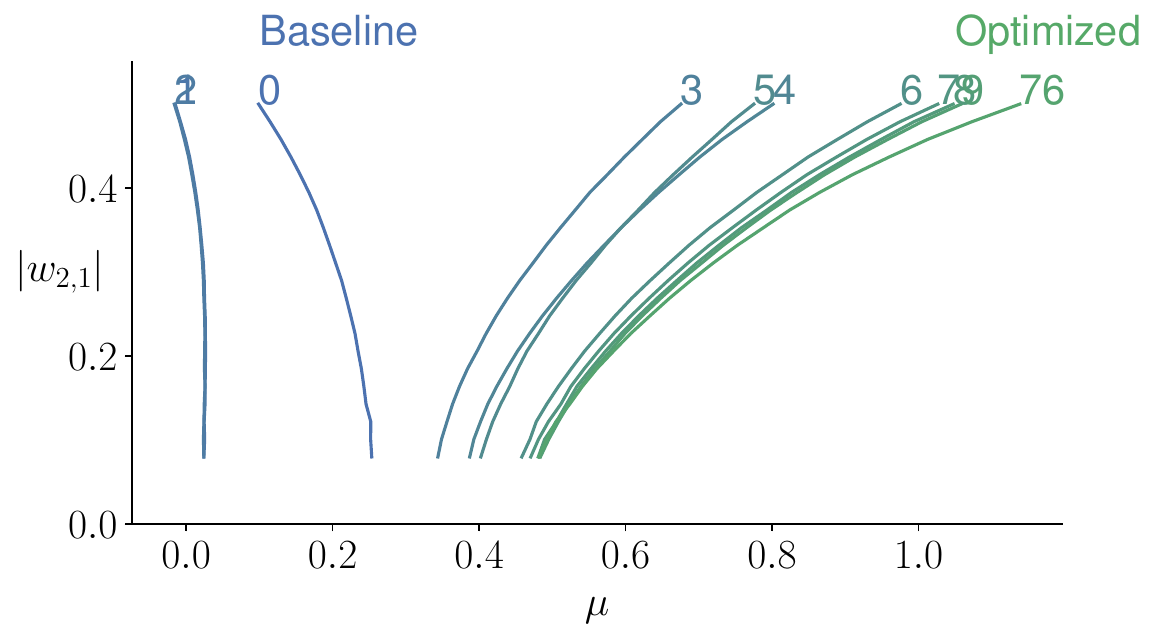}
\caption{Optimization history of the CGL system showing the LCO curve going from subcritical to supercritical.}
\label{fig:opt_cgl}
\end{figure}

\begin{figure}[ht]
\centering
\includegraphics[width=0.9\linewidth]{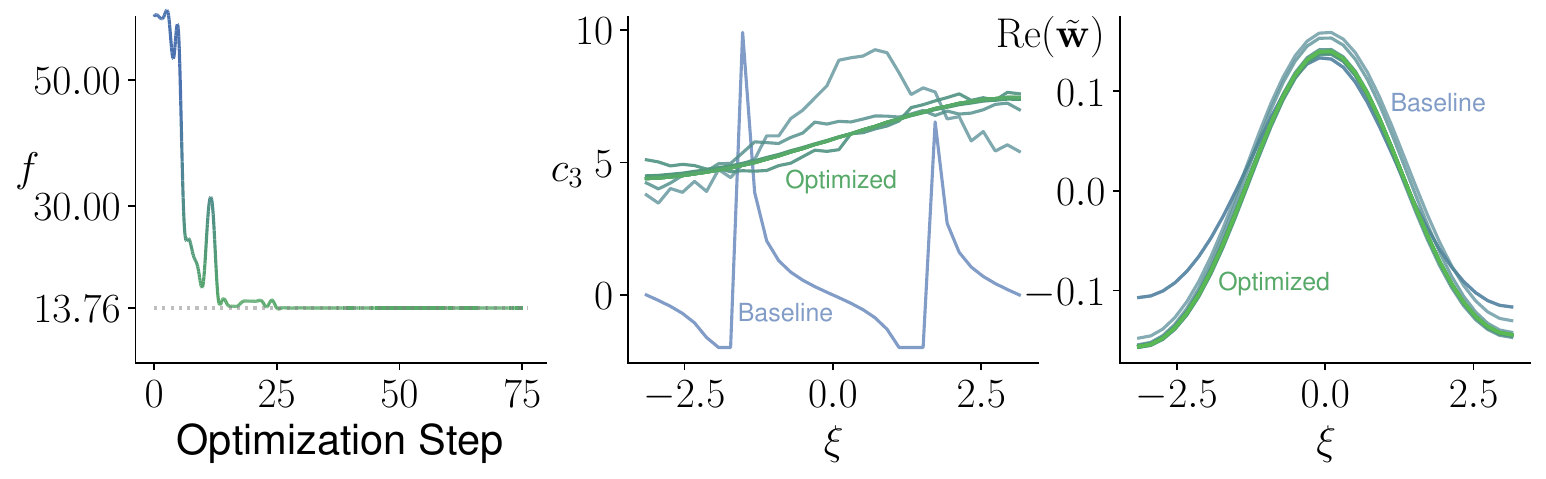}
\caption{Optimization history of the objective function (left), design variable (middle), and real part of the equilibrium solution (right).}
\label{fig:opthist_cgl}
\end{figure}

\begin{figure}[ht]
\centering
\includegraphics[width=0.99\linewidth]{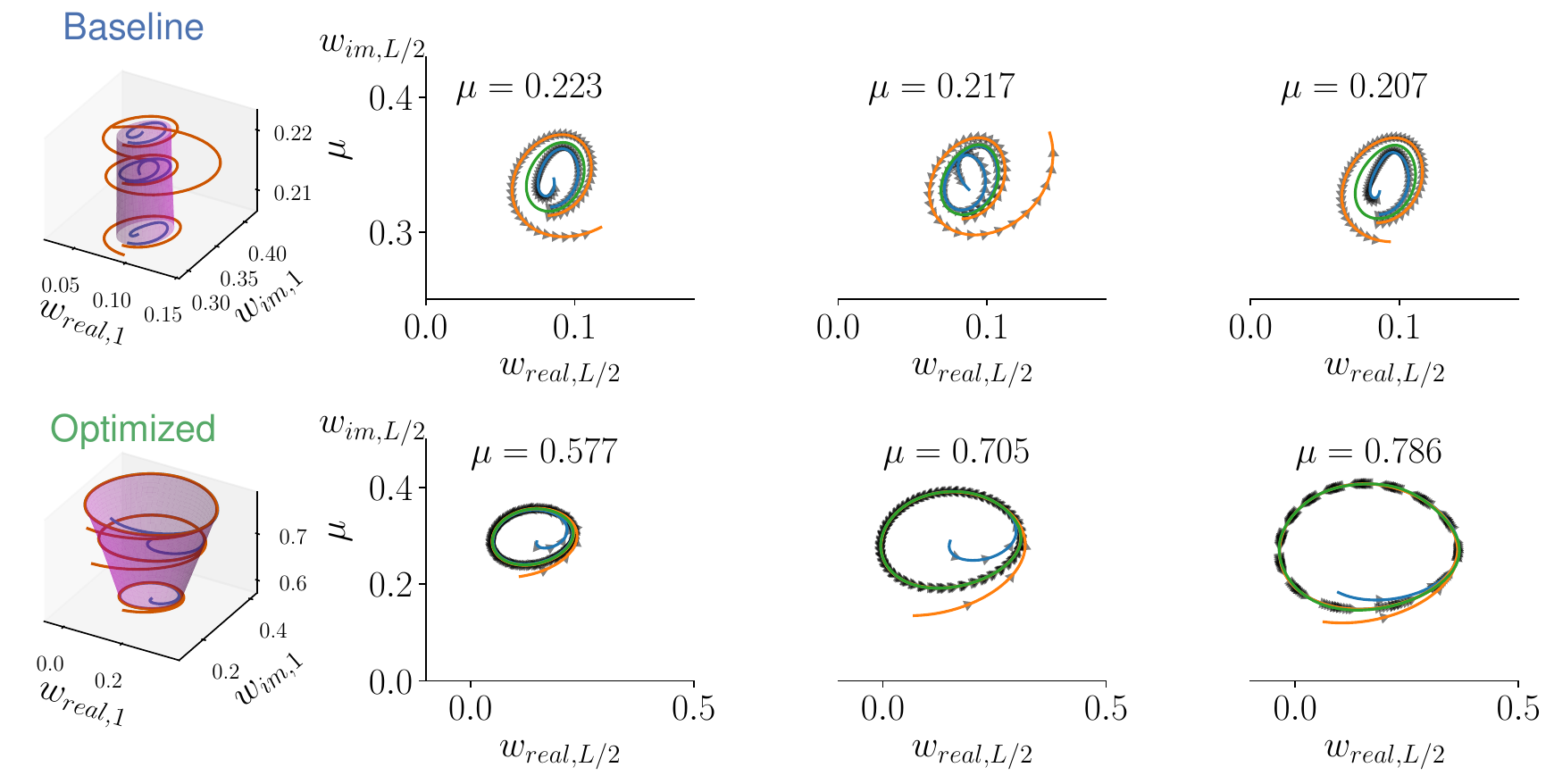}
\caption{Baseline (top) and optimized (bottom) response for the CGL example (Eq.~\ref{eq:optimization_CGL}). The baseline design includes unstable subcritical bifurcation, while the optimized design has only stable supercritical bifurcation and LCOs.}
\label{fig:cgl_ex}
\end{figure}

\section{Conclusion}
\label{sec:con}
This paper proposes the first Lyapunov coefficient-based method for Hopf-bifurcation stability analysis and optimization.
{The advantage of this method is that it rigorously accounts for the stability of bifurcations—quantifying whether the emerging limit cycle is stable or unstable—through the first Lyapunov coefficient.}
{In addition, the method enables efficient computation of the stability metric's derivatives, which are essential for gradient-based optimization.}
This is achieved by solving two large adjoint equations derived from the global coupled adjoint equation and two smaller adjoint equations obtained from the partial derivatives of the function of interest.
The process is independent of the number of design variables, making it suitable for design optimization with high-dimensional design spaces.
Another advantage of this approach is that it eliminates the need to solve LCO curves under multiple operating conditions.
Instead, only one point in the parameter space needs to be analyzed by solving coupled EVPs.
Further computational tests are required to quantify the exact cost savings.

{Two main limitations of this method have been identified: (1) The first Lyapunov coefficient is a local method for predicting bifurcation instability and cannot predict the stability of a large magnitude LCO.
(2) The proposed method cannot handle changing bifurcation instability mechanisms during the optimization. 
These limitations may result in designs which still exhibit unstable Hopf-bifurcations due to a large magnitude LCO or a changing bifurcation stability mode during the optimization.
This should be considered in future work to ensure robust design optimization.}

One challenge in deriving this algorithm lies in the complicated form taken by the first Lyapunov coefficient, which involves matrix inversion, higher-order partial derivative matrix--vector products, and complex analytic function operation.
We differentiate this function and derive its RAD form using the dot product test-based method for complex analytic functions.
We verified the proposed derivative formulation of the first Lyapunov coefficient, demonstrating good agreement of the total adjoint derivative and partial derivatives compared to a finite-difference reference, with a small relative difference around $10^{-5}$ and less for both. 

We demonstrate the effectiveness of the constraint and its derivative computation in a gradient-based design optimization.
We solve {three} optimization problems: an analytic problem, a nonlinear aeroelastic design problem {of a typical section model}, {and a problem governed by the complex Ginzburg--Landau PDE.}
In all cases, the proposed constraint suppresses the bifurcation present in the baseline design, finding an optimal solution in a few iterations.
The method has the potential to be used in aerodynamic and aerostructural optimization based on computational fluid dynamics, as well as other engineering applications governed by PDEs.

\bibliography{references,mdolab}

\appendix

\section{Eigenvector normalization}
\label{sec:normalization}
Here we show that sign of the first Lyapunov coefficient is invariant when the following eigenvector normalization condition in Eq.~\ref{eq:l_eig} is chosen (repeated here for convenience),
\begin{equation}
{\mb{q}_R}_{\mathrm{bif}}^* {\mb{q}_L}_{\mathrm{bif}} = 1.
\end{equation}
To demonstrate this, suppose with some arbitrary normalization condition for ${\mb{q}_R}_{\mathrm{bif}}$, we have
\begin{equation}
\label{eq:normalization}
{{\mb{q}_R}_{\mathrm{bif}}^{(0)}}^{*} {\mb{q}_L}_{\mathrm{bif}}^{(0)} = 1.
\end{equation}
where $\square^{(0)}$ denotes the original eigenvectors before stretching and rotation.
With some other normalization condition, we end up with a different ${\mb{q}_R}_{\mathrm{bif}}$, defined by
\begin{equation}
{\mb{q}_R}_{\mathrm{bif}} = \alpha e^{j\theta}{\mb{q}_R}_{\mathrm{bif}}^{(0)}.
\end{equation}
The eigenvector ${\mb{q}_L}_{\mathrm{bif}}$ needs to change accordingly,
\begin{equation}
{\mb{q}_L}_{\mathrm{bif}} = \beta e^{j\phi}{\mb{q}_L}_{\mathrm{bif}}^{(0)}.
\end{equation}
Using Eq.~\ref{eq:normalization}, we have 
\begin{equation}
\begin{aligned}
{\mb{q}_R}_{\mathrm{bif}}^* {\mb{q}_L}_{\mathrm{bif}}
= & \alpha\beta e^{j(\phi-\theta)}{{\mb{q}_R}_{\mathrm{bif}}^{(0)}}^{*}{\mb{q}_L}_{\mathrm{bif}}^{(0)} \\
= & \alpha\beta e^{j(\phi-\theta)}                                                                     \\
= & 1.                                                                                                 \\
\end{aligned}
\end{equation}
Now, by the fact that $\mb{b}(\cdot, \cdot)$ and $\mb{c}(\cdot, \cdot, \cdot)$ are linear with each of their entry, utilizing Eq.~\ref{eq:lyp} and Eq.~\ref{eq:bc} we have that
\begin{equation}
\begin{aligned}
h_1({\mb{q}_R}_{\mathrm{bif}}, {\mb{q}_L}_{\mathrm{bif}}) = & h_1({{\mb{q}_R}_{\mathrm{bif}}^{(0)}}, {\mb{q}_L}_{\mathrm{bif}}^{(0)})\alpha^3\beta e^{j(\theta - \phi)} = \alpha^2h_1({{\mb{q}_R}_{\mathrm{bif}}^{(0)}}, {\mb{q}_L}_{\mathrm{bif}}^{(0)}), \\
h_2({\mb{q}_R}_{\mathrm{bif}}, {\mb{q}_L}_{\mathrm{bif}}) = & h_2({{\mb{q}_R}_{\mathrm{bif}}^{(0)}}, {\mb{q}_L}_{\mathrm{bif}}^{(0)})\alpha^3\beta e^{j(\theta - \phi)} = \alpha^2h_2({{\mb{q}_R}_{\mathrm{bif}}^{(0)}}, {\mb{q}_L}_{\mathrm{bif}}^{(0)}), \\
h_3({\mb{q}_R}_{\mathrm{bif}}, {\mb{q}_L}_{\mathrm{bif}}) = & h_3({{\mb{q}_R}_{\mathrm{bif}}^{(0)}}, {\mb{q}_L}_{\mathrm{bif}}^{(0)})\alpha^3\beta e^{j(\theta - \phi)} = \alpha^2h_3({{\mb{q}_R}_{\mathrm{bif}}^{(0)}}, {\mb{q}_L}_{\mathrm{bif}}^{(0)}). \\
\end{aligned}
\end{equation}
Thus, we have
\begin{equation}
l({\mb{q}_R}_{\mathrm{bif}}, {\mb{q}_L}_{\mathrm{bif}}) = \alpha^2 l({\mb{q}_R}_{\mathrm{bif}}^{(0)}, {\mb{q}_L}_{\mathrm{bif}}^{(0)}).
\end{equation}
Since $\alpha\neq 0$, we have the sign of the first Lyapunov coefficient unchanged with arbitrary normalization condition for ${\mb{q}_R}_{\mathrm{bif}}$ as long as we normalize ${\mb{q}_L}_{\mathrm{bif}}$ using Eq.~\ref{eq:normalization}.

\section{Derivative of first Lyapunov coefficient, $f_{\text{Lyp}}$}
\label{sec:partial_l}
Here we develop the partial derivative of the first Lyapunov exponent, $f_{\text{Lyp}}$, with respect to $(\mb{x},  \omega_{\text{bif}}, \overtilde{\mb{w}}_{\text{bif}}, \mu_{\text{bif}}, \mb{q}_{R, \text{bif}}, \mb{q}_{L, \text{bif}})$.
To keep the notation clear we use the following notation $(\mb{x},  \omega, \overtilde{\mb{w}}, \mu, \mb{q}, \mb{p})$ to simplify it.
The derivatives require differentiation of the second-order residual derivative, $\mb{b}$, and the third-order residual derivative, $\mb{c}$, given in Eq.~\ref{eq:bc}, with respect to its multiple inputs.
For example, when we compute the derivative of $\mb{b}=\mb{b}_i(\mb{y}_1, \mb{y}_2, \overtilde{\mb{w}}, \mu, \mb{x})$ with respect to the second input, $\mb{y}_2$, evaluated at, $(\mb{y}_1, \mb{y}_2, \overtilde{\mb{w}}, \mu, \mb{x})=(\mb{s}_1, \mb{s}_2, \mb{s}_3, s_4, \mb{s}_5)$, it is written as
\begin{equation}
\f{\p \mb{b}}{\p \mb{y}_2}^*\Big\vert_{\mb{s}_1, \mb{s}_2, \mb{s}_3, s_4, \mb{s}_5}.
\end{equation}
For $\mb{c}$, we follow the same convention.

Following it, we present the partial derivative results
\begin{equation}
\begin{aligned}
\f{\p l}{\p \mb{x}} =        & \f{\p \mb{b}}{\p \mb{x}}^*\Big\vert_{\mb{q}, \mb{q}_1, \overtilde{\mb{w}}, \mu, \mb{x}} \overbar{\mb{b}}_1 + \f{\p \mb{b}}{\p \mb{x}}^*\Big\vert_{\mathrm{Conj}(\mb{q}), \mb{q}_2, \overtilde{\mb{w}}, \mu, \mb{x}} \overbar{\mb{b}}_2                                                                                                                                                                                                       \\
& + \f{\p \mb{b}}{\p \mb{x}}^*\Big\vert_{\mb{q}, \mathrm{Conj}(\mb{q}), \overtilde{\mb{w}}, \mu, \mb{x}} \overbar{\mb{b}}_3 + \f{\p \mb{b}}{\p \mb{x}}^*\Big\vert_{\mb{q}, \mb{q}, \overtilde{\mb{w}}, \mu, \mb{x}} \overbar{\mb{b}}_4                                                                                                                                                                                                         \\
& + \f{\p \mb{c}}{\p \mb{x}}^*\Big\vert_{\mb{q}, \mb{q}, \mathrm{Conj}(\mb{q}), \overtilde{\mb{w}}, \mu, \mb{x}} \overbar{\mb{c}} + \f{\p \mb{A}}{\p \mb{x}}^* \overbar{\mb{A}},                                                                                                                                                                                                                                                               \\
\f{\p l}{\p \tilde{\mb{w}}}= & \f{\p \mb{b}}{\p \tilde{\mb{w}}}^*\Big\vert_{\mb{q}, \mb{q}_1, \overtilde{\mb{w}}, \mu, \mb{x}} \overbar{\mb{b}}_1 + \f{\p \mb{b}}{\p \tilde{\mb{w}}}^*\Big\vert_{\mathrm{Conj}(\mb{q}), \mb{q}_2, \overtilde{\mb{w}}, \mu, \mb{x}} \overbar{\mb{b}}_2                                                                                                                                                                                       \\
& + \f{\p \mb{b}}{\p \tilde{\mb{w}}}^*\Big\vert_{\mb{q}, \mathrm{Conj}(\mb{q}), \overtilde{\mb{w}}, \mu, \mb{x}} \overbar{\mb{b}}_3 + \f{\p \mb{b}}{\p \tilde{\mb{w}}}^*\Big\vert_{\mb{q}, \mb{q}, \overtilde{\mb{w}}, \mu, \mb{x}} \overbar{\mb{b}}_4                                                                                                                                                                                         \\
& + \f{\p \mb{c}}{\p \tilde{\mb{w}}}^*\Big\vert_{\mb{q}, \mb{q}, \mathrm{Conj}(\mb{q}), \overtilde{\mb{w}}, \mu, \mb{x}} \overbar{\mb{c}} + \f{\p \mb{A}}{\p \tilde{\mb{w}}}^* \overbar{\mb{A}}                                                                                                                                                                                                                                                \\
\f{\p l}{\p \mb{q}} =        & \left(\f{\p \mb{b}}{\p \mb{y}_1}^*\Big\vert_{\mb{q}, \mb{q}_1, \overtilde{\mb{w}}, \mu, \mb{x}} \overbar{\mb{b}}_1\right) + \left(\mathrm{Conj}\left(\f{\p \mb{b}}{\p \mb{y}_1}^*\Big\vert_{\mathrm{Conj}(\mb{q}), \mb{q}_2, \overtilde{\mb{w}}, \mu, \mb{x}} \overbar{\mb{b}}_2\right)\right)                                                                                                                                               \\
& + \left(\f{\p \mb{b}}{\p \mb{y}_1}^*\Big\vert_{\mb{q}, \mathrm{Conj}(\mb{q}), \overtilde{\mb{w}}, \mu, \mb{x}} \overbar{\mb{b}}_3 + \mathrm{Conj}\left(\f{\p \mb{b}}{\p \mb{y}_1}^*\Big\vert_{\mb{q}, \mathrm{Conj}(\mb{q}), \overtilde{\mb{w}}, \mu, \mb{x}} \overbar{\mb{b}}_3\right)\right)                                                                                                                                               \\
& + \left(\f{\p \mb{b}}{\p \mb{y}_1}^*\Big\vert_{\mb{q}, \mb{q}, \overtilde{\mb{w}}, \mu, \mb{x}} \overbar{\mb{b}}_4 + \f{\p \mb{b}}{\p \mb{y}_1}^*\Big\vert_{\mb{q}, \mb{q}, \overtilde{\mb{w}}, \mu, \mb{x}} \overbar{\mb{b}}_4\right)                                                                                                                                                                                                       \\
& + \left(\f{\p \mb{c}}{\p \mb{y}_1}^*\Big\vert_{\mb{q}, \mb{q}, \mathrm{Conj}(\mb{q}), \overtilde{\mb{w}}, \mu, \mb{x}} \overbar{\mb{c}} + \f{\p \mb{c}}{\p \mb{y}_1}^*\Big\vert_{\mb{q}, \mb{q}, \mathrm{Conj}(\mb{q}), \overtilde{\mb{w}}, \mu, \mb{x}} \overbar{\mb{c}} + \mathrm{Conj}\left(\f{\p \mb{c}}{\p \mb{y}_1}^*\Big\vert_{\mb{q}, \mb{q}, \mathrm{Conj}(\mb{q}), \overtilde{\mb{w}}, \mu, \mb{x}} \overbar{\mb{c}}\right)\right) \\
\f{\p l}{\p \mb{p}} =        & \f{1}{2\omega}\left(\mb{c} - 2\mb{b}_1 + \mb{b}_2\right)                                                                                                                                                                                                                                                                                                                                                                                     \\
\f{\p l}{\p {\mu}} =         & \f{\p \mb{b}}{\p \mu}^*\Big\vert_{\mb{q}, \mb{q}_1, \overtilde{\mb{w}}, \mu, \mb{x}} \overbar{\mb{b}}_1 + \f{\p \mb{b}}{\p \mu}^*\Big\vert_{\mathrm{Conj}(\mb{q}), \mb{q}_2, \overtilde{\mb{w}}, \mu, \mb{x}} \overbar{\mb{b}}_2                                                                                                                                                                                                             \\
& + \f{\p \mb{b}}{\p \mu}^*\Big\vert_{\mb{q}, \mathrm{Conj}(\mb{q}), \overtilde{\mb{w}}, \mu, \mb{x}} \overbar{\mb{b}}_3 + \f{\p \mb{b}}{\p \mu}^*\Big\vert_{\mb{q}, \mb{q}, \overtilde{\mb{w}}, \mu, \mb{x}} \overbar{\mb{b}}_4                                                                                                                                                                                                               \\
& + \f{\p \mb{c}}{\p \mu}^*\Big\vert_{\mb{q}, \mb{q}, \mathrm{Conj}(\mb{q}), \overtilde{\mb{w}}, \mu, \mb{x}} \overbar{\mb{c}} + \f{\p \mb{A}}{\p \mu}^* \overbar{\mb{A}}                                                                                                                                                                                                                                                                      \\
\f{\p l}{\p {\omega}}  =     & -\f{1}{2\omega^2}\mathrm{Re}\left(\mb{p}^*\mb{c} - 2\mb{p}^* \mb{b}_1 + \mb{p}^* \mb{b}_2\right) + (2j)^*\left(\left(2j\omega\mb{I} - \mb{A}_2\right)^{-1}\mb{b}_4\right)^* \left(-\overbar{\mb{b}}_4\right)\\
\end{aligned},
\end{equation}
where
\begin{equation}
\begin{aligned}
\overbar{\mb{A}} =   & \left(-\left(\mb{A}\right)^{-*}\f{\p \mb{b}}{\p \mb{y}_2}^*\Big\vert_{\mb{q}, \mb{q}_1, \overtilde{\mb{w}}, \mu, \mb{x}} \overbar{\mb{b}}_1\right)\left(\mb{A}^{-1}\mb{b}_3\right)^* + \overbar{\mb{b}}_4\left(\left(2j\omega\mb{I} - \mb{A}\right)^{-1}\mb{b}_4\right)^* \\
\overbar{\mb{b}}_1 = & -\f{1}{\omega}\mb{p}                                                                                                                                                                                                                                                      \\
\overbar{\mb{b}}_2 = & \f{1}{2\omega}\mb{p}                                                                                                                                                                                                                                                      \\
\overbar{\mb{b}}_3 = & \mb{A}^{-*}\f{\p \mb{b}}{\p \mb{y}_2}^*\Big\vert_{\mb{q}, \mb{q}_1, \overtilde{\mb{w}}, \mu, \mb{x}} \overbar{\mb{b}}_1                                                                                                                                                   \\
\overbar{\mb{b}}_4 = & \left(2j\omega\mb{I} - \mb{A}\right)^{-*}\f{\p \mb{b}}{\p \mb{y}_2}^*\Big\vert_{\mathrm{Conj}(\mb{q}), \mb{q}_2, \overtilde{\mb{w}}, \mu, \mb{x}} \overbar{\mb{b}}_2                                                                                                      \\
\overbar{\mb{c}} =   & \f{1}{2\omega}\mb{p}                                                                                                                                                                                                                                                      \\
\end{aligned},
\end{equation}
and
\begin{equation}
\begin{aligned}
\mb{A} =   & \mb{A}(\overtilde{\mb{w}}, \mu, \mb{x})                                                                            \\
\mb{b}_1 = & \mb{b}(\mb{q}, \mb{A}^{-1}\mb{b}_3, \overtilde{\mb{w}}, \mu, \mb{x})                                               \\
\mb{b}_2 = & \mb{b}(\mathrm{Conj}(\mb{q}), \left(2j\omega\mb{I} - \mb{A}\right)^{-1}\mb{b}_4,  \overtilde{\mb{w}}, \mu, \mb{x}) \\
\mb{b}_3 = & \mb{b}(\mb{q}, \mathrm{Conj}(\mb{q}), \overtilde{\mb{w}}, \mu, \mb{x})                                             \\
\mb{b}_4 = & \mb{b}(\mb{q}, \mb{q}, \overtilde{\mb{w}}, \mu, \mb{x}),                                                           \\
\mb{c} =   & \mb{c}(\mb{q}, \mb{q}, \mathrm{Conj}(\mb{q}), \overtilde{\mb{w}}, \mu, \mb{x})                                     \\
\mb{q}_1 = & \mb{A}_1^{-1}\mb{b}_3                                                                                              \\
\mb{q}_2 = & \left(2j\omega_2\mb{I} - \mb{A}_2\right)^{-1}\mb{b}_4                                                              \\
\end{aligned}.
\end{equation}
It is worthwhile noting that although the derivative computation seems complex, the computational cost is close to the primal analysis.
All partial derivative RAD can be computed relatively at a low cost, and the more expensive computations are:
\begin{equation}
\begin{aligned}
& \mb{A}^{-*}\left(\f{\p \mb{b}}{\p \mb{y}_2}^*\Big\vert_{\mb{q}, \mb{q}_1, \overtilde{\mb{w}}, \mu, \mb{x}} \overbar{\mb{b}}_1\right)                                              \\
& \left(2j\omega\mb{I} - \mb{A}\right)^{-*}\left(\f{\p \mb{b}}{\p \mb{y}_2}^*\Big\vert_{\mathrm{Conj}(\mb{q}), \mb{q}_2, \overtilde{\mb{w}}, \mu, \mb{x}} \overbar{\mb{b}}_2\right) \\
\end{aligned},
\end{equation}
which are the adjoint equations of the primal analysis,
\begin{equation}
\label{sec:lyp_adjoint}
\begin{aligned}
& \mb{A}^{-1}\mb{b}_3                               \\
& \left(2j\omega\mb{I} - \mb{A}\right)^{-1}\mb{b}_4 \\
\end{aligned},
\end{equation}
respectively.
In addition, the dense matrix, $\overbar{\mb{A}}$, is written as the sum of two outer products.
The matrix is never be stored explicitly, instead, the two vectors are stored, saving memory.

\section{Adjoint equation for $\mb{r}_L$}
\label{sec:adjoint_L}
The first row of Eq.~\ref{eq:seq_adjoint} can be further expanded as
\begin{equation}
\begin{bmatrix}
\mb{A}_L^\intercal                                   & {\mb{G}_L}_{r}^\intercal                 & {\mb{G}_L}_{i}^\intercal                & \mb{O}                        & \mb{O}                         \\
\mb{O}                                               & \mb{A}_L                                 & {\omega_L}_{\mathrm{bif}}\mb{I}         & {\mb{q}_{R}}_{\mathrm{bif},r} & -{\mb{q}_{R}}_{\mathrm{bif},i} \\
\mb{O}                                               & - {\omega_L}_{\mathrm{bif}}\mb{I}        & \mb{A}_L                                & {\mb{q}_{R}}_{\mathrm{bif},i} & {\mb{q}_{R}}_{\mathrm{bif},r}  \\
\f{\p \mb{r}}{\p {{\mu}_L}_{\mathrm{bif}}}^\intercal & {\mb{g}_L}_{r}^\intercal                 & {\mb{g}_L}_{i}^\intercal                & \mb{O}                        & \mb{O}                         \\
\mb{O}                                               & -{\mb{q}_{L}}_{\mathrm{bif},i}^\intercal & {\mb{q}_{L}}_{\mathrm{bif},r}^\intercal & \mb{O}                        & \mb{O}
\end{bmatrix}
\pmb{\psi}_L
=
\begin{bmatrix}
\f{\p f}{\p {\overtilde{\mb{w}}_L}_{\mathrm{bif}}}^\intercal \\
\f{\p f}{\p {\mb{q}_L}_{\mathrm{bif}, r}}^\intercal          \\
\f{\p f}{\p {\mb{q}_L}_{\mathrm{bif}, i}}^\intercal          \\
\f{\p f}{\p {\mu_L}_{\mathrm{bif}}}^\intercal                \\
\f{\p f}{\p {\omega_L}_{\mathrm{bif}}}^\intercal             \\
\end{bmatrix},
\end{equation}
Here, ${\mb{G}_L}_{r}$ and ${\mb{G}_L}_i$ are defined by
\begin{equation}
\begin{aligned}
{\mb{G}_L}_{r} [i_1, i_2] = & \sum_{i_3}\f{\p^2 \mb{r}[i_3] {\mb{q}_{L}}_{\mathrm{bif},r}[i_3]}{\p {\overtilde{\mb{w}}_L}_{\mathrm{bif}}[i_1] \p {\overtilde{\mb{w}}_L}_{\mathrm{bif}}[i_2]}, \\
{\mb{G}_L}_{i} [i_1, i_2] = & \sum_{i_3}\f{\p^2 \mb{r}[i_3] {\mb{q}_{L}}_{\mathrm{bif},i}[i_3]}{\p {\overtilde{\mb{w}}_L}_{\mathrm{bif}}[i_1] \p {\overtilde{\mb{w}}_L}_{\mathrm{bif}}[i_2]}, \\
\end{aligned}
\end{equation}
where $[i_1, i_2]$ indicates getting the element at $i_1^{\mathrm{th}}$ row and $i_2^{\mathrm{th}}$ column, and ${\mb{g}_L}_{r}$ and ${\mb{g}_L}_{i}$ are defined by
\begin{equation}
\begin{aligned}
{\mb{g}_L}_{r}[i_1] = & \sum_{i_2}\f{\p^2 \mb{r}[i_2] {\mb{q}_{L}}_{\mathrm{bif},r}[i_2]}{\p {\overtilde{\mb{w}}_L}_{\mathrm{bif}}[i_1] \p {\mu_L}_{\mathrm{bif}}}, \\
{\mb{g}_L}_{i}[i_1] = & \sum_{i_2}\f{\p^2 \mb{r}[i_2] {\mb{q}_{L}}_{\mathrm{bif},i}[i_2]}{\p {\overtilde{\mb{w}}_L}_{\mathrm{bif}}[i_1] \p {\mu_L}_{\mathrm{bif}}}. \\
\end{aligned}
\end{equation}
Also, notice that
\begin{equation}
\mb{A}_L^\intercal = \mb{A}^\intercal({\mu_L}_{\mathrm{bif}}, {\overtilde{\mb{w}}_L}_{\mathrm{bif}}, \mb{x}),
\end{equation}
i.e., $\mb{A}_L$ is the Jacobian matrix.

\section{Adjoint equation for $\mb{r}_R$}
\label{sec:adjoint_R}
The second row of Eq.~\ref{eq:seq_adjoint} can be further expanded as
\begin{equation}
\begin{bmatrix}
\mb{A}_R^\intercal                                   & {\mb{G}_R}^\intercal_{r}                & {\mb{G}_R}^\intercal_{i}                 & \mb{O}                         & \mb{O}   \\
\mb{O}                                               & \mb{A}_R^\intercal                      & -{\omega}_{\mathrm{bif}}\mb{I}         & 2{\mb{q}_{R}}_{\mathrm{bif},r} & \mb{O}   \\
\mb{O}                                               & {\omega}_{\mathrm{bif}}\mb{I}         & \mb{A}_R^\intercal                       & 2{\mb{q}_{R}}_{\mathrm{bif},i} & \mb{e}_j \\
\f{\p \mb{r}}{\p {{\mu}_R}_{\mathrm{bif}}}^\intercal & {\mb{g}_R}_{r}^\intercal                & {\mb{g}_R}_{i}^\intercal                 & \mb{O}                         & \mb{O}   \\
\mb{O}                                               & {\mb{q}_{R}}_{\mathrm{bif},i}^\intercal & -{\mb{q}_{R}}_{\mathrm{bif},r}^\intercal & \mb{O}                         & \mb{O}
\end{bmatrix}
\pmb{\psi}_R
=
\begin{bmatrix}
\f{\p f}{\p {\overtilde{\mb{w}}}_{\mathrm{bif}}}^\intercal - \f{\p \mb{r}_L}{\p {\overtilde{\mb{w}}}_{\mathrm{bif}}}^\intercal \boldsymbol{\psi}_L \\
\f{\p f}{\p {\mb{q}_R}_{\mathrm{bif}, r}}^\intercal - \f{\p \mb{r}_L}{\p {\mb{q}_R}_{\mathrm{bif}, r}}^\intercal \boldsymbol{\psi}_L                   \\
\f{\p f}{\p {\mb{q}_R}_{\mathrm{bif}, i}}^\intercal - \f{\p \mb{r}_L}{\p {\mb{q}_R}_{\mathrm{bif}, i}}^\intercal \boldsymbol{\psi}_L                   \\
\f{\p f}{\p {\mu}_{\mathrm{bif}}}^\intercal - \f{\p \mb{r}_L}{\p {\mu}_{\mathrm{bif}}}^\intercal \boldsymbol{\psi}_L                               \\
\f{\p f}{\p {\omega}_{\mathrm{bif}}}^\intercal - \f{\p \mb{r}_L}{\p {\omega}_{\mathrm{bif}}}^\intercal \boldsymbol{\psi}_L                         \\
\end{bmatrix},
\end{equation}
where ${\mb{G}_{R}}_r$ and ${\mb{G}_{R}}_i$ are defined by
\begin{equation}
\begin{aligned}
{\mb{G}_{R}}_{r}[i_1, i_2] = & \sum_{i_3}\f{\p^2 \mb{r}[i_1] {\mb{q}_R}_{\mathrm{bif}, r}[i_3]}{\p {\overtilde{\mb{w}}}_{\mathrm{bif}}[i_3] \p {\overtilde{\mb{w}}}_{\mathrm{bif}}[i_2]}, \\
{\mb{G}_{R}}_{i}[i_1, i_2] = & \sum_{i_3}\f{\p^2 \mb{r}[i_1] {\mb{q}_R}_{\mathrm{bif}, i}[i_3]}{\p {\overtilde{\mb{w}}}_{\mathrm{bif}}[i_3] \p {\overtilde{\mb{w}}}_{\mathrm{bif}}[i_2]}, \\
\end{aligned}
\end{equation}
and ${\mb{g}_R}_{r}$ and ${\mb{g}_R}_{i}$ are defined by
\begin{equation}
\begin{aligned}
{\mb{g}_R}_{r}[i_1] = & \sum_{i_3}\f{\p^2 \mb{r}[i_1] {\mb{q}_{L}}_{\mathrm{bif},r}[i_3]}{\p {\overtilde{\mb{w}}_L}_{\mathrm{bif}[i_3]} \p {\mu}_{\mathrm{bif}}}, \\
{\mb{g}_R}_{i}[i_1] = & \sum_{i_2}\f{\p^2 \mb{r}[i_1] {\mb{q}_{L}}_{\mathrm{bif},i}[i_3]}{\p {\overtilde{\mb{w}}_L}_{\mathrm{bif}}[i_3] \p {\mu}_{\mathrm{bif}}}. \\
\end{aligned}
\end{equation}
Also, notice that
\begin{equation}
\mb{A}_R^\intercal = \mb{A}^\intercal({\mu}_{\mathrm{bif}}, {\overtilde{\mb{w}}}_{\mathrm{bif}}, \mb{x}).
\end{equation}
i.e., $\mb{A}_R$ is the Jacobian matrix.
In addition, for the contribution to the right hand side (RHS) from the adjoint of the previous equation, i.e., terms for ${\p \mb{r}_L}/{\p {\overtilde{\mb{w}}}_{\mathrm{bif}}}^\intercal \boldsymbol{\psi}_L, {\p \mb{r}_L}/{\p {\mb{q}_R}_{\mathrm{bif}, r}}^\intercal \boldsymbol{\psi}_L, {\p \mb{r}_L}/{\p {\mb{q}_R}_{\mathrm{bif}, i}}^\intercal \boldsymbol{\psi}_L,  {\p \mb{r}_L}/{\p {\mu}_{\mathrm{bif}}}^\intercal \boldsymbol{\psi}_L, {\p \mb{r}_L}/{\p {\omega}_{\mathrm{bif}}}^\intercal \boldsymbol{\psi}_L$, most of the terms are zero, except ${\p \mb{r}_L}/{\p {\mb{q}_R}_{\mathrm{bif}, r}}^\intercal \boldsymbol{\psi}_L, {\p \mb{r}_L}/{\p {\mb{q}_R}_{\mathrm{bif}, i}}^\intercal\boldsymbol{\psi}_L$ due to the fact that the normalization condition of the left EVP relies on the knowledge of the right EVP (see Eq.~\ref{eq:l_eig}).

\end{document}